\documentclass[a4paper,reqno,11pt]{amsart}
\usepackage{amsfonts,amsmath,amsthm,amssymb,stmaryrd,mathrsfs}
\newtheorem{theorem}{Theorem}[section]
\newtheorem{lemma}[theorem]{Lemma}
\newtheorem{Proposition}[theorem]{Proposition}
\newtheorem{Remark}[theorem]{Remark}

\usepackage{color}
\numberwithin{equation}{section}
\allowdisplaybreaks

\usepackage[left=1 in, right=1 in,top=1 in, bottom=1 in]{geometry}

\providecommand{\commu}[1]{\left[#1\right]}

\DeclareMathOperator{\diver}{div}

\newcommand{\na}{\nabla}

\newcommand{\la}{\lambda}
\newcommand{\de}{\delta}

\newcommand{\pa}{\partial}

\newcommand{\eps}{\epsilon}
\newcommand{\va}{\varepsilon}

\newcommand{\De}{\Delta}

\providecommand{\norm}[1]{\left\Vert#1\right\Vert}

\def\r3{\mathbb{R}^3}

\def\ls{\lesssim}
\def\gt{\gtrsim}
\def\le{\leqslant}
\def\ge{\geqslant}

\begin{document}
\title[Stability of Navier-Stokes-Poisson equations]{Stability of steady states of the Navier-Stokes-Poisson equations with non-flat doping profile}

\author{Zhong Tan}
\address{School of Mathematical Sciences\\
Xiamen University\\
Xiamen, Fujian 361005, China}
\email[Z. Tan]{ztan85@163.com}

\author{Yanjin Wang}
\address{School of Mathematical Sciences\\
Xiamen University\\
Xiamen, Fujian 361005, China}
\email[Y. J. Wang]{yanjin$\_$wang@xmu.edu.cn}

\author{Yong Wang}
\address{School of Mathematical Sciences\\
Xiamen University\\
Xiamen, Fujian 361005, China}
\email[Y. Wang]{wangyongxmu@163.com}



\thanks{Corresponding author: Yong Wang, wangyongxmu@163.com}
\thanks{Z. Tan was supported by the National Natural Science Foundation
of China (No. 11271305). Y. J. Wang was supported by the National Natural Science Foundation of China (No. 11201389), the
Natural Science Foundation of Fujian Province of China (No. 2012J05011), the Specialized Research Fund for
the Doctoral program of Higher Education (No. 20120121120023), and the Fundamental Research Funds for
the Central Universities (No. 2013121002).}

\begin{abstract}
We consider the stability of the steady state of the compressible Navier-Stokes-Poisson equations with the non-flat doping profile.
We prove the global existence of classical solutions near the steady state for the large doping profile. For the small doping profile, we prove the time decay rates of the solution provided that the initial perturbation belongs to $L^p$ with $1\le p< 3/2$.
\end{abstract}

\maketitle

\section{Introduction}\label{section1}
%%%%%%%%%%%%%%%%%%%%%%%%%%%%%%%%%%%%%%%%%%%%%%
The dynamics of charged particles of one carrier type (e.g., electrons) can be described by the compressible Navier-Stokes-Poisson equations:
\begin{align}  \label{NSP}
\begin{cases}
\displaystyle\partial_t\rho+\diver  (\rho u)=0,   \\
\displaystyle\partial_t(\rho u)+\diver (\rho u\otimes u)+\na p(\rho)-\mu\Delta u-(\mu+\mu')\na\diver u=\rho\na\phi, \\
\Delta\phi=\rho-b,
\\(\rho,u)\mid_{t=0}=(\rho_0, u_0).
\end{cases}
\end{align}
Here  $\rho=\rho(t,x),  u=u(t,x)$ represent the density and velocity functions of the electrons respectively, at time $t\ge0$ and position $x\in \mathbb{R}^3$. The pressure $p =p (\rho )$ is a smooth function  with
$p '(\rho )>0$ for $\rho >0$. We assume that the constant viscosity coefficients $\mu$ and $\mu'$ satisfy the usual physical conditions
\begin{align}\label{viscosity}
\mu>0,\quad\mu'+\frac23\mu\ge0.
\end{align}
The self-consistent electric potential $\phi=\phi(t,x)$ is coupled with the density through the Poisson equation, where the function $b=b(x)$ is the doping profile for the ions. We assume that $b$ is a smooth function satisfying
\begin{align}\label{doping condition}
b(x)>0,\quad \lim_{|x|\to+\infty}b(x)=\bar{b}>0.
\end{align}

For the pressure law $p(\rho)=\rho^\gamma$ with the adiabatic exponent $\gamma> {3}/{2}$, the global existence of weak solutions was obtained by \cite{KS2008} when the spatial dimension is three in the framework of Lions-Feireisl for the compressible Navier-Stokes equations \cite{L,FNP}. This result was later extended by \cite{TW} to the case $\gamma>1$ when the dimension is two, where the authors introduced an idea to overcome the new difficulty caused by that the Poisson term $\rho\nabla\phi$ may not be integrable when $\gamma$ is close to one. The large-time behavior of weak solutions towards the steady state was also considered in \cite{TW}. When the doping profile is flat, $i.e.$, $b(x)=\bar b$, the steady state of \eqref{NSP} is the trivial constant one $(\bar \rho,0,0)$ with $\bar\rho=\bar b$. The unique global solution around this constant state in $H^N$ was proved by \cite{LMZ} in the framework of Matsumura-Nishida for the Navier-Stokes equations \cite{MN1}. Moreover, when the initial perturbation $(\rho_0-\bar\rho,u_0)$ is small in $L^1$, the $L^2$ decay of the solution to \eqref{NSP} was also obtained in \cite{LMZ}:
\begin{align}\label{LNSP}
\norm{(\rho-\bar\rho)(t)}_{L^2}
\lesssim(1+t)^{-\frac{3}{4} }\ \text{ and }\norm{u  (t)}_{L^2}\lesssim(1+t)^{-\frac{1}{4} }.
\end{align}
This may imply that the presence of the electric field slows down the time decay rate of the velocity with the factor $1/2$ compared to the Navier-Stokes equations \cite{MN2,P}. It was proved by transforming the system \eqref{NSP} into the Navier-Stokes equations with a non-local force
\begin{align}
\begin{cases}
\partial_t\rho+{\rm div}(\rho  u )=0
\\\partial_t(\rho  u) +{\rm div}(\rho  u  \otimes u )+\nabla p(\rho)-\mu\Delta u -(\mu+\mu')\na \diver u =\rho\nabla\Delta^{-1}(\rho-\bar\rho).
\end{cases}
\end{align}
However, the author in \cite{W} gave a different (contrary) comprehension of the effect of the electric field on the time decay rates of the solution. When the initial perturbation $(\rho_0-\bar\rho,u_0,\nabla\phi_0)$ belongs to $L^p$ with $1<p\le 2$, the $L^2$ decay of the solution to \eqref{NSP} was obtained in \cite{W}:
\begin{align}
\norm{(\rho-\bar\rho)(t)}_{L^2}
\ls(1+t)^{-\frac{3}{2}\left(\frac{1}{p}-\frac{1}{2}\right)-\frac{1}{2}}\ \text{and }\norm{u  (t)}_{L^2}\ls(1+t)^{-\frac{3}{2}\left(\frac{1}{p}-\frac{1}{2}\right)}.
\end{align}
In this sense, the electric field enhances the time decay rate of the density with the factor $1/2$! This can be understood well from the physical point of view since we get an additional dispersive effect from the repulsive electric force.

In this paper, we will study the asymptotic stability of the steady state of the system \eqref{NSP} with the non-flat doping profile $b(x)$.
A steady state $(\rho_s,\phi_s)$ with $u_s\equiv0$ of \eqref{NSP} must satisfy
\begin{align}  \label{state}
\begin{cases}
\displaystyle\na p(\rho_s)=\rho_s\na\phi_s, \\
\Delta\phi_s=\rho_s-b.
\end{cases}
\end{align}
We will record the existence and uniqueness of the  solution to \eqref{state} in Proposition \ref{prop}.

\smallskip\smallskip
\noindent \textbf{Notations.}
We use $L^{p}(\mathbb{R}^{3})$, $1\le p\le
\infty $ to denote the $L^{p}$ spaces with norm $
\norm{\cdot}_{L^{p}}$, and $W^{k,p}(\mathbb{R}^{3})$ to denote the usual
Sobolev spaces with norm $\norm{\cdot}_{W^{k,p}}$, and $H^k=W^{k,2}$. $\na ^{\ell }$ with $\ell \in \mathbb{R}$ stands for the usual spatial derivatives of order $
\ell $; we allow that $\ell <0$ or $\ell $ is not a positive integer.

Throughout this paper, we let $C$ denote the universal positive constants. We will use $A \lesssim B$ if $A \le C B$ and $A \gt B$ if $A \ge C B$, and we may write $\frac{d}{dt}A+B\ls D$ for $\frac{d}{dt}A+CB\lesssim D$. For simplicity, we write $\norm{(A,B)}_{X}:=\norm{A}_{X}+\norm{B}_{X}$ and $\int f:=\int_{\r3}f\, dx$.

Our first main result of the global solutions to \eqref{NSP} near the steady state for the large doping profile is stated as the following theorem.
\begin{theorem}\label{existence}
Assume that $\na b(x)\in H^k$ with $k\ge 3$, and $(\rho_s ,\phi_s)$ of \eqref{state} is constructed in Proposition \ref{prop}. If $\norm{(\rho_0-\rho_s,u_0)}_{H^k}+\norm{\nabla^{-1}(\rho_0-\rho_s)}_{L^2}$ is sufficiently small,
then there exists a unique global solution $(\rho ,u,\na\phi )$ to the system \eqref{NSP} such that for all $t\ge0$,
\begin{align}\label{energy inequlity}
&\norm{(\rho-\rho_s,u)(t)}_{H^k}^{2} +\norm{(\na\phi-\na\phi_s)(t)}_{L^2}^{2} +\int_{0}^{t}\left(
\norm{(\rho-\rho_s)(\tau)}_{H^{k}}^{2}+\norm{\na u(\tau)}
_{H^k}^{2}\right) d\tau \nonumber\\
&\quad\le C\left(\norm{(\rho_0-\rho_s, u_0)}_{H^k}^2+\norm{\nabla^{-1}(\rho_0-\rho_s)}_{L^2}^2\right).
\end{align}
\end{theorem}

\begin{Remark}
In Theorem \ref{existence}, the initial condition on the smallness of $\norm{\nabla^{-1}(\rho_0-\rho_s)}_{L^2}$ is required so that the initial potential energy $\norm{\na\phi_0-\na\phi_s}_{L^2}$ is small. Indeed, since $ \phi-\phi_s $ satisfies the Poisson equation $\Delta(\phi-\phi_s )=\rho-\rho_s$, we have
$$\norm{\na\phi_0-\na\phi_s}_{L^2}=\norm{\na\De^{-1}(\rho_0-\rho_s)}_{L^2}=\norm{\nabla^{-1}(\rho_0-\rho_s)}_{L^2}.$$
Such condition can be guaranteed by that, for instance, $\norm{\rho_0-\rho_s}_{L^{6/5}}$ is small.
\end{Remark}

Our second main result of the time decay rates of the solution to \eqref{NSP} towards the steady state for the small doping profile is stated as the next theorem.
\begin{theorem}\label{decay}
Assume that the assumptions of Theorem \ref{existence} hold for $k\ge4$. If $\norm{\na b}_{H^k}+\norm{b-\bar b}_{L^r}$ with $1< r<3/2$ is sufficiently small and $\norm{(\nabla^{-1}(\rho_0-\rho_s),u_0)}_{L^p}$ with $1\le p<3/2$ is finite,
then for $0\le\ell\le 1/2$,
\begin{align}\label{decay11}
\norm{\na^{\ell}(\rho-\rho_s)(t) }_{H^{k-\ell}}\le C_0(1+t)^{-\frac{3}{2}\left(\frac{1}{\max\{p,r\}}-\frac{1}{2}\right)-\frac{\ell}{2}-\frac{1}{2}}
\end{align}
and for $0\le\ell\le 3/2$,
\begin{align}\label{decay1}
\norm{\na^\ell u (t)}_{H^{k-\ell}}\le C_0(1+t)^{-\frac{3}{2}\left(\frac{1}{\max\{p,r\}}-\frac{1}{2}\right)-\frac{\ell}{2}},
\end{align}
and
\begin{align}\label{decay2}
\norm{ (\rho-\rho_s,u )(t)}_{L^\infty}\le C_0(1+t)^{-\frac{3}{2}\left(\frac{1}{\max\{p,r\}}-\frac{1}{2}\right)-\frac{3}{4}},
\end{align}
where $C_0$ is a positive constant depending on the initial data.
\end{theorem}

\begin{Remark}
In Theorem \ref{decay}, the initial condition on the boundedness of $\norm{\nabla^{-1}(\rho_0-\rho_s)}_{L^p}$ with $1\le p<3/2$ is required.
If $p>1$, then such condition can be guaranteed by that, for instance, $\rho_0-\rho_s={\rm div} f$ for some $f\in L^{p}$, thanks to the singular integral theory \cite{S}.
\end{Remark}

\begin{Remark}
Theorems \ref{existence} and \ref{decay} extend the previous results of \cite{W}. Note that the argument of proving the time decay of the solution in \cite{W}, which follows a pure energy method introduced in \cite{GW}, highly depends on that the doping profile is flat. To show the time decay of the solution for the non-flat doping profile, we shall need to employ a different argument as explained below.
\end{Remark}

\begin{Remark}
In \cite{DLUY,DUYZ1}, the authors proved the time decay of the compressible Navier-Stokes equations with an external potential force provided that the initial perturbation belongs to $L^p$ with $ 1\le p<6/5$. We may expect to employ our arguments of proving Theorem \ref{decay} to extend the range of $p$ therein to be $1\le p<3/2$. The key point is to introduce the fractional derivatives in the study of the time decay as already seen from \eqref{decay11}--\eqref{decay1}.
\end{Remark}

Theorems \ref{existence} and \ref{decay} will be proved in Sections \ref{section3}--\ref{section4}, respectively. To prove Theorem \ref{existence}, we will reformulate the system \eqref{NSP} into \eqref{NSP per} for the perturbation $\varrho=\rho-\rho_s$ and $u=u$.
To derive the energy estimates, the difficulty is caused by the terms on the left-hand side of \eqref{NSP per}. More precisely, we can not directly control as in \cite{HJW} the terms resulting from when the differential operator $\nabla^l$ commutates with the functions of $\rho_s$ since $\rho_s$ may not be close to a positive constant in the current case. To overcome this difficulty, we first notice that these commutator terms do not appear when $l=0$, which allows us to derive the zero-order energy estimates as stated in Lemma \ref{En le 1}.
When $l=1,2,\dots,k$, we will carry out the delicate analysis so that we can control these commutator terms with an (large) error term  as stated in Lemma \ref{En le 2}, briefly speaking, $\norm{\nabla u}_{L^2}$. However, this error has been controlled by the previous step. Hence, after recovering the dissipation estimates of $\varrho$ by Lemmas \ref{En le 3}--\ref{En le 4}, we can close the energy estimates.

To prove Theorem \ref{decay}, we will reformulate \eqref{NSP per} into the system \eqref{NSP per2} with constant coefficients. Since $\rho_s$ is close to $\bar\rho$, we can improve the energy estimates in the proof of Theorem \ref{existence} to deduce the estimates \eqref{2proof1}, which implies that the decay of $\norm{\na^{\ell-1} \varrho (t)}_{H^{k+1-\ell}}^2+\norm{\na^\ell u (t)}_{H^{k-\ell}}^2$ can be obtained from the decay of $\norm{ \na^{\ell-1}\varrho(t)}_{L^2}^2+\norm{ \na^{\ell}u(t)}_{L^2}^2+\norm{(\varrho,u)(t)}_{L^\infty}^2$ for $0\le \ell\le 3/2$. On the other hand, using the linear decay estimates of the linear Navier-Stokes-Poisson equations with constant coefficients, we can derive the estimates conversely.
This interplay would then be closed by the smallness of the solution and the doping profile. Finally, we may remark that if we do not introduce the fractional derivatives, then we can only choose $\ell=0,1$ in \eqref{2proof1}, which would result that we could only prove Theorem \ref{decay} for $1\le p<6/5$.

The rest of this paper is organized as follows. In Section \ref{section2}, we will prove the existence of the stationary solution to \eqref{state}. We will prove Theorem \ref{existence} and Theorem \ref{decay} in Section \ref{section3} and Section \ref{section4}, respectively. Some analytic
tools will be collected in Appendix.

%%%%%%%%%%%%%%%%%%%%%%%%%%%%%%%%%%%%%%%%%%%%%%
\section{Steady state}\label{section2}
%%%%%%%%%%%%%%%%%%%%%%%%%%%%%%%%%%%%%%%%%%%%%%

In this section, we record the following existence and uniqueness of the solutions to \eqref{state}.
\begin{Proposition}\label{prop}
Assume that $b(x)$ is a smooth function satisfying \eqref{doping condition}. Then there exists a unique classical solution $(\rho_s,\phi_s)$ to \eqref{state}. Moreover,
\begin{itemize}
\item $\rho_s$ has the positive upper and lower bounds, $i.e.,$
\begin{align}\label{b1}
0<\inf_{x\in\r3}b(x)\le \rho_s(x)\le \sup_{x\in\r3}b(x)<\infty;
\end{align}
\item if $\na b\in H^k$ with $k\ge 3$, then there exists a constant $C$ depending on $\norm{\na b}_{H^k}$ such that
 \begin{align}\label{b2}
\norm{\na \rho_s}_{H^k}\le C;
\end{align}

\item if $\norm{\na b}_{H^k}$ is sufficiently small, then
\begin{align}\label{123}
\norm{\na \rho_s}_{H^k}\ls\norm{\na b}_{H^k};
\end{align}

\item if further $
\norm{b-\bar b}_{L^r}<\infty $
with $1<r<\infty$, then
\begin{align}\label{Lr estimate}
\norm{\rho_s-\bar \rho}_{W^{2,r}}\ls \norm{b-\bar b}_{L^r},
\end{align}
where $\bar\rho=\bar b$.
\end{itemize}

\end{Proposition}
\begin{proof}
The existence and uniqueness of the classical solutions to \eqref{state} satisfying the first three assertions were proved in \cite{HJW} for $k=3$, but the case $k\ge4$ can be handled in the same way and so we omit the proof. We may then focus on proving the last assertion {\it a priori}.

To this end, setting $h'(s)=p'(s)/s$, we deduce from \eqref{state} that
\begin{align}\label{eee}
\diver \left(h'(\rho_s)\na\rho_s\right)=\rho_s-b.
\end{align}
Writing $f=\rho_s-\bar{\rho}$ with $\bar\rho=\bar b$, we rewrite \eqref{eee} as
\begin{align}\label{ddd}
-h'(\bar \rho)\De f+f=\diver \left(\left(h'(\rho_s)-h'(\bar \rho)\right)\na f\right)+b-\bar{b}.
\end{align}
It then follows from the standard elliptic theory \cite{S} on \eqref{ddd} that for $1<r<\infty$,
\begin{align}
\norm{f}_{W^{2,r}}&\ls \norm{\diver \left(\left(h'(\rho_s)-h'(\bar \rho)\right)\na f\right)+b-\bar{b}}_{L^r}\nonumber
\\&\ls \norm{\na \rho_s}_{H^k}\norm{f}_{W^{2,r}} +\norm{b-\bar{b}}_{L^r}.
\end{align}
This implies \eqref{Lr estimate} since $\norm{\na \rho_s}_{H^k}$ is small by \eqref{123}.
\end{proof}

%%%%%%%%%%%%%%%%%%%%%%%%%%%%%%%%%%%%%%%%%%%%%%
\section{Global solution with large doping profile}\label{section3}
%%%%%%%%%%%%%%%%%%%%%%%%%%%%%%%%%%%%%%%%%%%%%%

In this section, we will construct the global solutions near the steady state to \eqref{NSP} for the large doping profile.
For this, we define the perturbation by
\begin{align}
\varrho=\rho-\rho_s,\ u=u,\ \Phi=\phi-\phi_s.
\end{align}
In order to reformulate the problem \eqref{NSP} properly, we introduce the enthalpy function
\begin{align}
h(z)=\int_1^z \frac{p'(s)}{s}\, ds.
\end{align}
We also introduce the Taylor expansion
\begin{align}\label{hhh}
h(\varrho+\rho_s)=h(\rho_s)+h'(\rho_s)\varrho+\mathcal{R},
\end{align}
where the remainder $\mathcal{R}$ is given by
\begin{align}\label{hhhggg}
\mathcal{R}=\int_{\rho_s}^{\varrho+\rho_s}h''(s)(\varrho+\rho_s-s)\,ds.
\end{align}
Then the problem \eqref{NSP} can be reformulated into the perturbed form of
\begin{align} \label{NSP per}
\begin{cases}
\displaystyle\pa_t\varrho+\diver(\rho_s u)=-\diver(\varrho  u),   \\
\displaystyle\pa_tu+\na(h'(\rho_s)\varrho) -\frac{1}{\rho_s}\left(\mu\Delta u+(\mu+\mu')\na\diver u\right)-\na\Phi\\
\displaystyle\quad=-u\cdot\na u-\na \mathcal{R} +\left(\frac{1}{\varrho+\rho_s}-\frac{1}{\rho_s}\right)\left(\mu\Delta u+(\mu+\mu')\na\diver u\right), \\
\Delta\Phi=\varrho, \\
(\varrho,u)\mid_{t=0}=(\varrho_0, u_0).
\end{cases}
\end{align}

%%%%%%%%%%%%%%%%%%%%%%%%%%%%%%%%%%%%%%%%%%%%%%%%%%%%%%%%%%%%%%%%%%%%%%%%%%%%%
\subsection{Energy estimates}
%%%%%%%%%%%%%%%%%%%%%%%%%%%%%%%%%%%%%%%%%%%%%%%%%%%%%%%%%%%%%%%%%%%%%%%%%%%%%

In this subsection, we will derive the a priori estimates for the solutions to the Navier-Stokes-Poisson equations \eqref{NSP per} by assuming that for sufficiently small $\delta>0$,
\begin{align}\label{a priori}
\norm{ (\varrho,u)(t)}_{H^k}+\norm{\na\Phi(t)}_{L^2}\le \delta.
\end{align}

We first derive the zero-order energy estimates for the solution itself.
\begin{lemma}\label{En le 1} It holds that
\begin{align}\label{energy 0}
\frac{d}{dt}\int\left(h'(\rho_s)\varrho^2+\rho_s|u|^2+|\na\Phi|^2\right)+\norm{\na u}_{L^2}^2
\ls\de\norm{ \varrho }_{L^2}^2.
\end{align}
\end{lemma}
\begin{proof}
Multiplying the first two equations in \eqref{NSP per} by $h'(\rho_s)\varrho$ and $\rho_s u$ respectively, summing up them and then integrating over $\r3$, we obtain
\begin{align} \label{yi 0jieu}
&\frac{1}{2}\frac{d}{dt}\int\left(h'(\rho_s)\varrho^2+\rho_s|u|^2\right)+\mu\norm{\na u}_{L^2}^2+\left(\mu+\mu'\right)\norm{\diver u}_{L^2}^2 -\int \rho_su\cdot\na\Phi\nonumber\\
&\quad=-\int \left(h'(\rho_s)\varrho\diver(\varrho u)+\rho_su\cdot(u\cdot\na u)+\rho_su\cdot\na\mathcal{R}\right)\nonumber\\
&\qquad+\int\rho_su\cdot\left(\frac{1}{\varrho+\rho_s}-\frac{1}{\rho_s}\right)\left(\mu\Delta u+(\mu+\mu')\na\diver u\right).
\end{align}
Here we have used the integration by parts to have the cancelation:
\begin{align}
\int&\left( h'(\rho_s)\varrho \diver(\rho_s u)+\rho_s u\cdot \na(h'(\rho_s)\varrho)\right)=0.
\end{align}

To estimate the Poisson term on the left-hand side of \eqref{yi 0jieu}, we integrate by parts by several times, use the first equation and the Poisson equation in $\eqref{NSP per}$, by H\"older's and Sobolev's inequalities and the a priori bound \eqref{a priori}, and employ the standard elliptic estimates $\norm{\nabla^k \Phi}_{L^2}\ls\norm{\nabla^{k-2}\varrho}_{L^2}$ for $k\ge 2$, to deduce
\begin{align}\label{san 0jieu}
-\int \rho_su\cdot\na\Phi&=\int\diver(\rho_su)\Phi=-\int\pa_t\varrho\Phi-\int\diver(\varrho u)\Phi\nonumber\\
&=-\int\pa_t\De\Phi\Phi+\int\varrho u\cdot\na\Phi=\frac12\frac{d}{dt}\int|\na\Phi|^2+\int\varrho u\cdot\na\Phi\nonumber\\
&\ge\frac12\frac{d}{dt}\int|\na\Phi|^2-C\norm{\varrho}_{L^2}\norm{u}_{L^3}\norm{\na\Phi}_{L^6}\nonumber\\
&\ge\frac12\frac{d}{dt}\int|\na\Phi|^2-C\de\norm{\varrho}_{L^2}^2.
\end{align}

We now estimate the terms on the right-hand side of \eqref{yi 0jieu}. By H\"{o}lder's and Sobolev's inequalities, \eqref{a priori} and \eqref{b1}, we obtain
\begin{align}
&-\int \left(h'(\rho_s)\varrho\diver(\varrho u)+\rho_su\cdot(u\cdot\na u)\right)\nonumber\\
&\quad+\int\rho_su\cdot\left(\frac{1}{\varrho+\rho_s}-\frac{1}{\rho_s}\right)\left(\mu\Delta u+(\mu+\mu')\na\diver u\right)\nonumber\\
&\quad\ls\norm{\varrho}_{L^2}\norm{\na\varrho}_{L^3}\norm{u}_{L^6}+\norm{\varrho}_{L^2}\norm{\varrho}_{L^\infty}\norm{\diver u}_{L^2}\nonumber\\
&\qquad+\norm{u}_{L^6}\norm{u}_{L^3}\norm{\na u}_{L^2}+\norm{u}_{L^6}\norm{\varrho}_{L^2}\norm{\na^2 u}_{L^3}\nonumber\\
&\quad\ls\de\left(\norm{\varrho}_{L^2}^2+\norm{\na u}_{L^2}^2\right).
\end{align}
For the remaining term involved with $\na\mathcal{R}$, we integrate by parts to have
\begin{align}\label{si 0jieu}
-\int \rho_su\cdot\na\mathcal{R}&=\int \na\rho_s\cdot u\mathcal{R}+\rho_s\diver u\mathcal{R}\nonumber\\
&\ls\norm{u}_{L^\infty}\norm{\varrho}_{L^2}^2+\norm{\diver u}_{L^2}\norm{\varrho}_{L^2}\norm{\varrho}_{L^\infty}\nonumber\\
&\ls\de\left(\norm{\varrho}_{L^2}^2+\norm{\na u}_{L^2}^2\right).
\end{align}
Here we have used the fact from \eqref{hhh}--\eqref{hhhggg} that $\mathcal{R}= O(\varrho^2)$.

Plugging the estimates \eqref{san 0jieu}--\eqref{si 0jieu} into \eqref{yi 0jieu}, since $\de$ is small, we then conclude \eqref{energy 0}.
\end{proof}

We then derive the energy estimates for the derivatives of the solution.
\begin{lemma}\label{En le 2}
For $l=1,\dots,k$, we have that for any $\va>0$,
\begin{align}\label{energy 1}
&\frac{d}{dt}\int\left(h'(\rho_s)|\na^l\varrho|^2+\rho_s|\na^lu|^2\right)+\norm{\na^{l+1}u}_{L^2}^2\nonumber\\
&\quad \ls (\delta+\varepsilon)\left( \norm{ \na^{l }\varrho }_{L^2}^2+ \norm{\varrho }_{L^2}^2\right )+C_\varepsilon\left( \norm{ \na^{l }u }_{L^2}^2+ \norm{\na u }_{L^2}^2\right ).
\end{align}
\end{lemma}
\begin{proof}
Applying $\na^l$ to the first two equations in \eqref{NSP per} and then multiplying the resulting identities by $h'(\rho_s)\na^l\varrho$ and $\rho_s\na^lu$ respectively, summing up them and then integrating over $\r3$, we obtain
\begin{align} \label{yi u}
&\frac{1}{2}\frac{d}{dt}\int\left(h'(\rho_s)|\na^l\varrho|^2+\rho_s|\na^lu|^2\right)+\int \left( h'(\rho_s) \na^l\varrho \na^l\diver(\rho_s u)+\rho_s\na^lu\cdot \na^l \na(h'(\rho_s)\varrho)\right)\nonumber\\
&\quad-\int\rho_s\na^lu\cdot \na^l\left(\frac{1}{\rho_s}\left(\mu\Delta u+(\mu+\mu')\na\diver u\right)\right)- \int \rho_s\na^lu\cdot\na^l\na\Phi\nonumber\\
&\quad=-\int h'(\rho_s)\na^l\varrho\na^l\diver(\varrho u)-\int\rho_s\na^lu\cdot\na^l(u\cdot\na u)-\int\rho_s\na^lu\cdot\na^l\na\mathcal{R}\nonumber\\
&\qquad+\int\rho_s\na^lu\cdot \na^l\left(\left(\frac{1}{\varrho+\rho_s}-\frac{1}{\rho_s}\right)\left(\mu\Delta u+(\mu+\mu')\na\diver u\right)\right).
\end{align}

First, we estimate the terms on the left-hand side of \eqref{yi u}.
By integrating by parts by several times and employing the commutator notation \eqref{commutator}, we have
\begin{align}\label{yi zuoyi}
&\int  \left(h'(\rho_s) \na^l\varrho \na^l\diver(\rho_s u)+\rho_s\na^lu\cdot \na^l \na(h'(\rho_s)\varrho)\right) \nonumber\\
&\quad= \int  \left(h'(\rho_s) \na^l\varrho \na^l\diver(\rho_s u)-\diver(\rho_s\na^lu) \na^l  (h'(\rho_s)\varrho)\right) \nonumber\\
&\quad= \int  \left(h'(\rho_s) \na^l\varrho  \diver\left(\commu{\na^l,\rho_s} u\right)+\rho_s\na^lu\cdot\na \commu{\na^l,h'(\rho_s)}\varrho\right)\nonumber\\
&\quad= \int  h'(\rho_s) \na^l\varrho \left( \commu{\na^l,\rho_s}\diver u+\commu{\na^l,\na \rho_s}  u\right)\nonumber\\
&\qquad+\int\rho_s\na^lu\cdot \left(\commu{\na^l,h'(\rho_s)}\na  \varrho+\commu{\na^l,\na h'(\rho_s)}  \varrho   \right)  .
\end{align}
For the second term on the right-hand side of \eqref{yi zuoyi}, we employ the commutator estimates \eqref{commutator estimate} to have: for $l=1$,
\begin{align}
&\int  h'(\rho_s) \na\varrho \commu{\na,\na \rho_s}  u  \ls \norm{\na \varrho}_{L^2}\norm{\commu{\na ,\na \rho_s}  u}_{L^2}\nonumber\\
&\quad\ls \norm{\na \varrho}_{L^2} \norm{ \na^2 \rho_s }_{L^3}\norm{   u }_{L^6}
\ls \norm{\na \varrho}_{L^2}  \norm{  \na u }_{L^2};
\end{align}
for $l\ge 2$,
\begin{align}
&\int  h'(\rho_s) \na^l\varrho \commu{\na^l,\na \rho_s}  u \ls \norm{\na^l\varrho}_{L^2}\norm{\commu{\na^l,\na \rho_s}  u}_{L^2}\nonumber\\
&\quad\ls \norm{\na^l\varrho}_{L^2}\left(\norm{ \na^2 \rho_s }_{L^3}\norm{ \na^{l-1} u }_{L^6}+\norm{ \na^{l+1} \rho_s }_{L^2}\norm{u}_{L^\infty}\right)\nonumber\\
&\quad\ls \norm{\na^l\varrho}_{L^2}\left( \norm{ \na^{l } u }_{L^2}+ \norm{\na u }_{H^1}\right )
\ls \norm{\na^l\varrho}_{L^2}\left( \norm{ \na^{l } u }_{L^2}+ \norm{\na u }_{L^2}\right ).
\end{align}
Note that we have used \eqref{b2} and the interpolation from Lemma \ref{A1}: $\norm{\na^2 u }_{L^2}\ls\norm{ \na^{l } u }_{L^2}+ \norm{\na u }_{L^2}$ for $l\ge 2$. We may then conclude that this term can be bounded by
$$\norm{\na^l\varrho}_{L^2}\left( \norm{ \na^{l } u }_{L^2}+ \norm{\na u }_{L^2}\right ).$$
We can derive the same bound for the first term that for $l=1$,
\begin{align}
&\int  h'(\rho_s) \na\varrho \commu{\na,\rho_s}\diver u\lesssim\norm{\na\varrho}_{L^2}\norm{\commu{\na,\rho_s}\diver u}_{L^2}\nonumber\\
&\quad\ls \norm{\na \varrho}_{L^2} \norm{ \na \rho_s }_{L^\infty}\norm{ \na u }_{L^2}
\ls \norm{\na \varrho}_{L^2}  \norm{  \na u }_{L^2};
\end{align}
for $l\ge 2$,
\begin{align}
&\int  h'(\rho_s) \na^l\varrho \commu{\na^l, \rho_s} \diver u \ls \norm{\na^l\varrho}_{L^2}\norm{\commu{\na^l,\rho_s}\diver u}_{L^2}\nonumber\\
&\quad\ls \norm{\na^l\varrho}_{L^2}\left(\norm{ \na \rho_s }_{L^\infty}\norm{ \na^{l-1} \diver u }_{L^2}+\norm{ \na^l \rho_s }_{L^6}\norm{\diver u}_{L^3}\right )\nonumber\\
&\quad\ls \norm{\na^l\varrho}_{L^2}\left( \norm{ \na^{l} u }_{L^2}+ \norm{\na u }_{H^1}\right )
\ls \norm{\na^l\varrho}_{L^2}\left( \norm{ \na^{l } u }_{L^2}+ \norm{\na u }_{L^2}\right ).
\end{align}
Similarly, for the third and fourth terms, we have: for $l=1$,
\begin{align}
&\int\rho_s\na u\cdot \left(\commu{\na,h'(\rho_s)}\na  \varrho+\commu{\na,\na h'(\rho_s)}  \varrho   \right)\nonumber\\
&\quad\ls \norm{\na u}_{L^2}\left(\norm{\commu{\na,h'(\rho_s)}\na\varrho}_{L^2}+\norm{\commu{\na,\na h'(\rho_s)}\varrho}_{L^2}\right)\nonumber\\
&\quad\ls \norm{\na u}_{L^2}\left(\norm{\na h'(\rho_s)}_{L^\infty}\norm{\na\varrho}_{L^2}+\norm{\na^2h'(\rho_s)}_{L^3}\norm{\varrho}_{L^6}\right) \ls\norm{\na u}_{L^2}\norm{\na\varrho}_{L^2};
\end{align}
for $l\ge 2$,
\begin{align}
&\int\rho_s\na^lu\cdot \left(\commu{\na^l,h'(\rho_s)}\na  \varrho+\commu{\na^l,\na h'(\rho_s)}  \varrho   \right)\nonumber\\
&\quad\ls \norm{\na^lu}_{L^2}\norm{\commu{\na^l,h'(\rho_s)}\na\varrho}_{L^2}+\norm{\na^lu}_{L^2}\norm{\commu{\na^l,\na h'(\rho_s)}\varrho}_{L^2}\nonumber\\
&\quad\ls \norm{\na^lu}_{L^2}\left(\norm{ \na h'(\rho_s)}_{L^\infty}\norm{ \na^{l}\varrho}_{L^2}+\norm{\na^l h'(\rho_s)}_{L^6}\norm{\na\varrho}_{L^3}\right)\nonumber\\
&\qquad+\norm{\na^lu}_{L^2}\left(\norm{ \na^2 h'(\rho_s)}_{L^3}\norm{ \na^{l-1}\varrho}_{L^6}+\norm{\na^{l+1} h'(\rho_s)}_{L^2}\norm{\varrho}_{L^\infty}\right)\nonumber\\
&\quad\ls \norm{\na^lu}_{L^2}\left( \norm{ \na^{l}\varrho}_{L^2}+ \norm{\na\varrho}_{H^1}\right )
\ls \norm{\na^lu}_{L^2}\left( \norm{ \na^{l}\varrho}_{L^2}+ \norm{\na\varrho}_{L^2}\right ).
\end{align}
Hence, we may bound the third and fourth terms by
$$\norm{\na^lu}_{L^2}\left( \norm{ \na^{l }\varrho }_{L^2}+ \norm{\na \varrho }_{L^2}\right ).$$
We may thus conclude that
\begin{align}\label{zuoyi}
&\int\left(h'(\rho_s) \na^l\varrho \na^l\diver(\rho_s u)+\rho_s\na^lu\cdot \na^l \na(h'(\rho_s)\varrho)\right)\nonumber\\
&\quad\gt-\left( \norm{ \na^{l }\varrho }_{L^2}+ \norm{\na\varrho }_{L^2}\right )\left( \norm{ \na^{l }u }_{L^2}+ \norm{\na u }_{L^2}\right ).
\end{align}

Next, we compute the following term
\begin{align}\label{div u yi}
&-\int\rho_s\na^lu\cdot \na^l\left(\frac{1}{\rho_s}   \na\diver u \right)=-\int\rho_s\na^lu\cdot \na^l \left(\na \left(\frac{1}{\rho_s}    \diver u \right)-\na\left(\frac{1}{\rho_s}   \right) \diver u \right)\nonumber\\
&\quad=\int\left(\diver\left(\rho_s\na^lu\right) \na^l \left(\frac{1}{\rho_s}    \diver u \right)+\rho_s\na^lu\cdot \na^l \left(\na\left(\frac{1}{\rho_s}   \right) \diver u \right)\right)\nonumber\\
&\quad=\int|\na^l\diver u |^2+\int \rho_s\na^l \diver u \commu{ \na^l ,\frac{1}{\rho_s}}\diver u \nonumber\\
&\qquad+\int\na \rho_s\cdot\na^lu\na^l \left(\frac{1}{\rho_s}\diver u \right)+\int\rho_s\na^lu\cdot \na^l \left(\na\left(\frac{1}{\rho_s} \right) \diver u \right).
\end{align}
For the second term on the right-hand side of \eqref{div u yi}, we employ the commutator estimates  \eqref{commutator estimate} to obtain: for $l=1$,
\begin{align}\label{div u er}
&\int\rho_s\na\diver u \commu{ \na,\frac{1}{\rho_s} }  \diver u\ls\norm{\na\diver u}_{L^2}\norm{\commu{ \na,\frac{1}{\rho_s} } \diver u}_{L^2}\nonumber\\
&\quad\ls\norm{\na^2 u}_{L^2}\norm{\na \left(\frac{1}{\rho_s}\right)}_{L^\infty}\norm{\diver u }_{L^2}\ls \norm{\na^2 u}_{L^2}\norm{\na u }_{L^2};
\end{align}
for $l\ge2$,
\begin{align}\label{div u san}
&\int\rho_s\na^l \diver u \commu{ \na^l ,\frac{1}{\rho_s} }  \diver u\ls\norm{\na^l\diver u}_{L^2}\norm{\commu{ \na^l ,\frac{1}{\rho_s} } \diver u}_{L^2}\nonumber\\
&\quad\ls\norm{\na^{l+1} u}_{L^2}\left(\norm{\na \left(\frac{1}{\rho_s}\right)}_{L^\infty}\norm{ \na^{l-1} \diver u }_{L^2}+\norm{ \na^l \left(\frac{1}{\rho_s}\right) }_{L^6}\norm{\diver u}_{L^3}\right )\nonumber\\
&\quad\ls \norm{\na^{l+1} u}_{L^2}\left( \norm{ \na^{l} u }_{L^2}+ \norm{\na u }_{H^1}\right )
\ls \norm{\na^{l+1} u}_{L^2}\left( \norm{ \na^{l } u }_{L^2}+ \norm{\na u }_{L^2}\right ).
\end{align}
For the third term, we employ the product estimates \eqref{product estimate} to obtain: for $l=1$,
\begin{align}\label{div u si}
&\int\na \rho_s\cdot\na u \na\left(\frac{1}{\rho_s}\diver u \right)\ls\norm{\na u}_{L^2}\norm{ \na \left(\frac{1}{\rho_s}\diver u \right)}_{L^2}\nonumber\\
&\quad\ls\norm{\na u}_{L^2}\left(\norm{\frac{1}{\rho_s}}_{L^\infty}\norm{\na \diver u}_{L^2}+\norm{\na \left(\frac{1}{\rho_s}\right)}_{L^3}\norm{\diver u}_{L^6}\right)\nonumber\\
&\quad\ls \norm{\na^2 u}_{L^2}\norm{\na u }_{L^2};
\end{align}
for $l\ge2$,
\begin{align}\label{div u wu}
&\int\na \rho_s\cdot\na^lu\cdot\na^l \left(\frac{1}{\rho_s}\diver u \right)\ls\norm{\na^lu}_{L^2}\norm{ \na^l \left(\frac{1}{\rho_s}\diver u \right)}_{L^2}\nonumber\\
&\quad\ls\norm{\na^lu}_{L^2}\left(\norm{\frac{1}{\rho_s}}_{L^\infty}\norm{\na^l\diver u}_{L^2}+\norm{\na^l\left(\frac{1}{\rho_s}\right)}_{L^6}\norm{\diver u}_{L^3}\right)\nonumber\\
&\quad\ls\norm{\na^lu}_{L^2}\left(\norm{\na^{l+1}u}_{L^2}+\norm{\na u}_{L^2}\right).
\end{align}
Similarly, for the fourth term, we have: for $l=1$,
\begin{align}\label{div u liu}
&\int\rho_s\na u\cdot \na \left(\na\left(\frac{1}{\rho_s} \right) \diver u \right)\ls\norm{\na u}_{L^2}\norm{\na \left(\na\left(\frac{1}{\rho_s} \right) \diver u \right)}_{L^2}\nonumber\\
&\quad\ls\norm{\na u}_{L^2}\left(\norm{\na\left(\frac{1}{\rho_s}\right)}_{L^\infty}\norm{\na \diver u}_{L^2}+\norm{\na^2 \left(\frac{1}{\rho_s}\right)}_{L^3}\norm{\diver u}_{L^6}\right)\nonumber\\
&\quad\ls \norm{\na^2 u}_{L^2}\norm{\na u }_{L^2};
\end{align}
for $l\ge2$,
\begin{align}\label{div u qi}
&\int\rho_s\na^lu\cdot \na^l \left(\na\left(\frac{1}{\rho_s} \right) \diver u \right)
\ls \norm{\na^lu}_{L^2}\norm{ \na^{l} \left(\na\left(\frac{1}{\rho_s} \right) \diver u \right)}_{L^2}\nonumber\\
&\quad\ls \norm{\na^lu}_{L^2} \left(\norm{\na\left(\frac{1}{\rho_s}\right)}_{L^\infty}\norm{\na^{l}\diver u}_{L^2}+\norm{\na^{l+1}\left(\frac{1}{\rho_s}\right)}_{L^2}\norm{\diver u}_{L^\infty}\right)\nonumber\\
&\quad\ls\norm{\na^lu}_{L^2}\left(\norm{\na^{l+1}u}_{L^2} +\norm{\na u}_{H^2}\right)\ls\norm{\na^lu}_{L^2}\left(\norm{\na^{l+1}u}_{L^2}+\norm{\na u}_{L^2}\right).
\end{align}
In light of \eqref{div u er}--\eqref{div u qi} and Cauchy's inequality, we deduce that the last three terms in \eqref{div u yi} are bounded by that for any small constant $\eta>0$,
$$\eta\norm{\na^{l+1}u}_{L^2}^2+C_\eta\left( \norm{ \na^{l }u }_{L^2}^2+ \norm{\na u }_{L^2}^2\right ).$$

We now compute the term
\begin{align}\label{Delta u yi}
-&\int\rho_s\na^lu\cdot \na^l\left(\frac{1}{\rho_s}   \De u \right)=-\int\rho_s\na^lu\cdot \na^l \left(\diver \left(\frac{1}{\rho_s}\na u \right)-\na\left(\frac{1}{\rho_s}\right)\cdot\na u \right)\nonumber\\
&=\int\na\left(\rho_s\na^lu\right)\cdot\na^l \left(\frac{1}{\rho_s}\na u \right)+\rho_s\na^lu\cdot \na^l \left(\na\left(\frac{1}{\rho_s}\right)\cdot\na u \right)\nonumber\\
&=\int|\na^{l+1}u |^2+\int \rho_s\na^{l+1}u\cdot\commu{ \na^l ,\frac{1}{\rho_s}}\na u \nonumber\\
&\quad+\int\na \rho_s\cdot\na^lu\cdot\na^l \left(\frac{1}{\rho_s}\na u\right)+\int\rho_s\na^lu\cdot \na^l \left(\na\left(\frac{1}{\rho_s}\right)\cdot\na u \right).
\end{align}
So as for \eqref{div u yi}, the last three terms in \eqref{Delta u yi} are also bounded by
$$\eta\norm{\na^{l+1}u}_{L^2}^2+C_\eta\left( \norm{ \na^{l }u }_{L^2}^2+ \norm{\na u }_{L^2}^2\right ).$$
Hence, we may conclude that
\begin{align}\label{zuosan}
&-\int\rho_s\na^lu\cdot \na^l\left(\frac{1}{\rho_s}\left(\mu\Delta u+(\mu+\mu')\na\diver u\right)\right)\nonumber
\\&\quad\gt\norm{\na^{l+1}u}_{L^2}^2-C \left( \norm{ \na^{l }u }_{L^2}^2+ \norm{\na u }_{L^2}^2\right ).
\end{align}

We may simply bound the Poisson term as
\begin{align}
- \int \rho_s\na^lu\cdot\na^l\na\Phi \ls \norm{\na^lu}_{L^2}\norm{\na^l\na\Phi}_{L^2}\ls \norm{\na^lu}_{L^2}\left(\norm{\varrho}_{L^2}+\norm{\na^{l-1}\varrho}_{L^2}\right).
\end{align}
Then, we obtain
\begin{align}\label{zuoer}
- \int \rho_s\na^lu\cdot\na^l\na\Phi\gt -\norm{\na^lu}_{L^2}\left(\norm{\varrho}_{L^2}+\norm{\na^{l-1}\varrho}_{L^2}\right).
\end{align}

Now, we estimate the nonlinear terms on the right-hand side of \eqref{yi u}. By the commutator notation \eqref{commutator}, we have
\begin{align}\label{yi youyi}
&-\int h'(\rho_s)\na^l\varrho\na^l\diver(\varrho u)=-\int h'(\rho_s)\na^l\varrho\na^l\left(u\cdot\na\varrho+\varrho\diver u\right)\nonumber\\
&\quad=-\int h'(\rho_s)u\cdot\na\na^l\varrho\na^l\varrho-\int h'(\rho_s)\na^l\varrho\left[\na^l,u\right]\cdot\na\varrho-\int h'(\rho_s)\na^l\varrho\na^l\left(\varrho\diver u\right).
\end{align}
By the integration by parts, we obtain
\begin{align}\label{er youyi}
&-\int h'(\rho_s)u\cdot\na\na^l\varrho\na^l\varrho=-\frac12\int h'(\rho_s)u\cdot\na|\na^l\varrho|^2=\frac12\int\diver(h'(\rho_s)u)|\na^l\varrho|^2\nonumber\\
&\quad\ls\left(\norm{u}_{L^\infty}+\norm{\diver u}_{L^\infty}\right)\norm{\na^l\varrho}_{L^2}^2\ls\de\norm{\na^l\varrho}_{L^2}^2.
\end{align}
By the commutator estimates \eqref{commutator estimate}, we obtain
\begin{align}\label{san youyi}
&-\int h'(\rho_s)\na^l\varrho\left[\na^l,u\right]\cdot\na\varrho\ls\norm{\na^l\varrho}_{L^2}\norm{\left[\na^l,u\right]\cdot\na\varrho}_{L^2}\nonumber\\
&\quad\ls\norm{\na^l\varrho}_{L^2}\left(\norm{\na u}_{L^\infty}\norm{\na^l\varrho}_{L^2}+\norm{\na^lu}_{L^6}\norm{\na\varrho}_{L^3}\right)\nonumber\\
&\quad\ls\de\left(\norm{\na^l\varrho}_{L^2}^2+\norm{\na^{l+1}u}_{L^2}^2\right).
\end{align}
By the product estimates  \eqref{product estimate}, we obtain
\begin{align}\label{si youyi}
&-\int h'(\rho_s)\na^l\varrho\na^l\left(\varrho\diver u\right)\ls\norm{\na^l\varrho}_{L^2}\norm{\na^l\left(\varrho\diver u\right)}_{L^2}\nonumber\\
&\quad\ls\norm{\na^l\varrho}_{L^2}\left(\norm{\varrho}_{L^\infty}\norm{\na^l\diver u}_{L^2}+\norm{\na^l\varrho}_{L^2}\norm{\diver u}_{L^\infty}\right)\nonumber\\
&\quad\ls\de\left(\norm{\na^l\varrho}_{L^2}^2+\norm{\na^{l+1}u}_{L^2}^2\right).
\end{align}
Hence, we deduce from \eqref{yi youyi}--\eqref{si youyi} that
\begin{align}\label{youyi}
-\int h'(\rho_s)\na^l\varrho\na^l\diver(\varrho u)\ls\de\left(\norm{\na^l\varrho}_{L^2}^2+\norm{\na^{l+1}u}_{L^2}^2\right).
\end{align}

We use the product estimates  \eqref{product estimate} to obtain
\begin{align}\label{youer}
&-\int\rho_s\na^lu\cdot\na^l(u\cdot\na u)\ls\norm{\na^lu}_{L^2}\norm{\na^l\left(u\cdot\na u\right)}_{L^2}\nonumber\\
&\quad\ls\norm{\na^lu}_{L^2}\left(\norm{u}_{L^\infty}\norm{\na^l\na u}_{L^2}+\norm{\na^lu}_{L^2}\norm{\na u}_{L^\infty}\right)\nonumber\\
&\quad\ls\de\left(\norm{\na^lu}_{L^2}^2+\norm{\na^{l+1}u}_{L^2}^2\right).
\end{align}

By the integration by parts, we deduce
\begin{align}\label{yousan}
&-\int\rho_s\na^lu\cdot\na^l\na\mathcal{R} =\int\na\rho_s\cdot\na^lu \na^{l}\mathcal{R}+\int\rho_s\na^{l}{\rm div }u \na^{l}\mathcal{R}\nonumber\\
&\quad\ls\norm{\na\rho_s}_{L^3}\norm{\na^lu}_{L^6}\norm{\na^{l}\mathcal{R}}_{L^{2}}+\norm{\na^{l+1}u}_{L^2}\norm{\na^{l}\mathcal{R}}_{L^2}\nonumber\\
&\quad\ls\de\left(\norm{\na\varrho}_{L^2}^2+\norm{\na^l\varrho}_{L^2}^2+\norm{\na^{l+1}u}_{L^2}^2\right).
\end{align}
Here we have used the nonlinear estimates \eqref{A31} of $\mathcal{R}$ stated in Lemma \ref{A3}.

For the last term, we integrate by parts to have
\begin{align}
&\int\rho_s\na^lu\cdot \na^l\left(\left(\frac{1}{\varrho+\rho_s}-\frac{1}{\rho_s}\right)\na^2 u\right)\nonumber
\\&\quad=-\int\na\rho_s\na^{l}u\cdot \na^{l-1}\left(\left(\frac{1}{\varrho+\rho_s}-\frac{1}{\rho_s}\right)\na^2 u \right)-\int\rho_s\na^{l+1}u\cdot \na^{l-1}\left(\left(\frac{1}{\varrho+\rho_s}-\frac{1}{\rho_s}\right)\na^2 u \right)\nonumber\\
&\quad\ls\norm{\na^{l}u}_{L^6}\norm{\na^{l-1}\left(\left(\frac{1}{\varrho+\rho_s}-\frac{1}{\rho_s}\right)\na^2 u\right)}_{L^{6/5}}\nonumber
\\&\qquad+\norm{\na^{l+1} u}_{L^2}\norm{\na^{l-1}\left(\left(\frac{1}{\varrho+\rho_s}-\frac{1}{\rho_s}\right)\na^2 u\right)}_{L^{2}} \nonumber\\
&\quad\ls\norm{\na^{l+1}u}_{L^2}\left(\norm{\frac{1}{\varrho+\rho_s}-\frac{1}{\rho_s}}_{L^3}\norm{\na^{l+1} u}_{L^2}+\norm{\na^{l-1}\left(\frac{1}{\varrho+\rho_s}-\frac{1}{\rho_s}\right)}_{L^2}\norm{\na^2 u}_{L^3}\right)\nonumber\\
&\qquad+\norm{\na^{l+1}u}_{L^2}\left(\norm{\frac{1}{\varrho+\rho_s}-\frac{1}{\rho_s}}_{L^\infty}\norm{\na^{l+1} u}_{L^2}+\norm{\na^{l-1}\left(\frac{1}{\varrho+\rho_s}-\frac{1}{\rho_s}\right)}_{L^6}\norm{\na^2 u}_{L^3}\right)\nonumber\\
&\quad\ls \de\left(\norm{\na^{l+1} u}_{L^2}^2+\norm{\na^{l-1}\varrho}_{L^2}^2+\norm{\na^l\varrho}_{L^2}^2\right) .
\end{align}
Hence, we deduce that
\begin{align}\label{qi yousi}
\int&\rho_s\na^lu\cdot \na^l\left(\left(\frac{1}{\varrho+\rho_s}-\frac{1}{\rho_s}\right)\left(\mu\Delta u+(\mu+\mu')\na\diver u\right)\right)\nonumber\\
&\ls \de\left(\norm{\na^{l+1} u}_{L^2}^2+\norm{\na^{l-1}\varrho}_{L^2}^2+\norm{\na^l\varrho}_{L^2}^2\right).
\end{align}

Plugging the estimates \eqref{zuoyi}, \eqref{zuosan}--\eqref{zuoer} and \eqref{youyi}--\eqref{qi yousi} into \eqref{yi u}, we then obtain
\begin{align}
&\frac{d}{dt}\left(\int h'(\rho_s)|\na^l\varrho|^2+\rho_s|\na^lu|^2\right)+\norm{\na^{l+1}u}_{L^2}^2\nonumber\\
&\quad\ls \delta\left( \norm{\na^{l+1} u}_{L^2}^2+\norm{\na^lu}_{L^2}^2+\norm{ \na^{l }\varrho }_{L^2}^2+\norm{ \na^{l-1 }\varrho }_{L^2}^2+ \norm{\na\varrho }_{L^2}^2+ \norm{\varrho }_{L^2}^2\right) \nonumber\\
&\qquad+  \norm{ \na^{l }u }_{L^2}^2+ \norm{\na u }_{L^2}^2 +\left( \norm{ \na^{l }u }_{L^2}+ \norm{\na u }_{L^2}\right )\left( \norm{ \na^{l }\varrho }_{L^2}+ \norm{\varrho }_{L^2}\right ).
\end{align}
By the interpolation and Young's inequality, since $\de$ is small, we deduce \eqref{energy 1}.
\end{proof}

Now we recover the dissipation for $\varrho$, and we first deal with $\varrho$ itself.

\begin{lemma} \label{En le 3}
It holds that
\begin{align}\label{energy 2}
\frac{d}{dt}\int u\cdot\na(h'(\rho_s)\varrho)+\norm{\varrho}_{L^2}^2+\norm{\na\varrho}_{L^2}^2\ls\norm{\na^2u}_{L^2}^2+\norm{\na u}_{L^2}^2.
\end{align}
\end{lemma}

\begin{proof}
Multiplying the second equation in $\eqref{NSP per}$ by $\na(h'(\rho_s)\varrho)$ and then integrating over $\r3$, by Cauchy's inequality and the nonlinear estimates \eqref{A31}, we obtain
\begin{align}\label{yi 0jie varrho}
&\int\pa_tu\cdot\na(h'(\rho_s)\varrho)+\norm{\na(h'(\rho_s)\varrho)}_{L^2}^2-\int\na\Phi\cdot\na(h'(\rho_s)\varrho)\nonumber\\
&\quad\ls\left(\norm{\na^2u}_{L^2}+\norm{u\cdot\na u}_{L^2}+\norm{\na\mathcal{R}}_{L^2}+\norm{\left(\frac{1}{\varrho+\rho_s}-\frac{1}{\rho_s}\right)\na^2u}_{L^2}\right)\norm{\na(h'(\rho_s)\varrho)}_{L^2}\nonumber\\
&\quad\ls \norm{\na^2u}_{L^2}^2+\delta\norm{\na\varrho}_{L^2}^2.
\end{align}

By the first equation in $\eqref{NSP per}$, we integrate by parts for both $t$- and $x$-variables to have
\begin{align}\label{san 0jie varrho}
&-\int\pa_tu\cdot\na(h'(\rho_s)\varrho)=-\frac{d}{dt}\int u\cdot\na(h'(\rho_s)\varrho)+\int u\cdot\na(h'(\rho_s)\pa_t\varrho)\nonumber\\
&\quad=-\frac{d}{dt}\int u\cdot\na(h'(\rho_s)\varrho)+\int h'(\rho_s)\diver u\left(u\cdot \na\varrho+ \varrho\diver  u+\rho_s\diver u+u\cdot\na \rho_s\right)\nonumber\\
&\quad\le -\frac{d}{dt}\int u\cdot\na(h'(\rho_s)\varrho)+C\norm{\na u}_{L^2}^2.
\end{align}
By the integration by parts and the Poisson equation in $\eqref{NSP per}$, we obtain
\begin{align}\label{er 0jie varrho}
-\int\na\Phi\cdot\na(h'(\rho_s)\varrho)=\int\De\Phi h'(\rho_s)\varrho=\int h'(\rho_s)\varrho^2.
\end{align}
Plugging \eqref{san 0jie varrho}--\eqref{er 0jie varrho} into \eqref{yi 0jie varrho}, since $\de$ is small, we deduce \eqref{energy 2} by noticing that
\begin{align}
\norm{\varrho}_{L^2}^2+\norm{\na\varrho}_{L^2}^2\ls\int h'(\rho_s)\varrho^2+\norm{\na(h'(\rho_s)\varrho)}_{L^2}^2.
\end{align}
We thus conclude the lemma.
\end{proof}

We next derive the dissipation estimates for the derivatives of $\varrho$.
\begin{lemma} \label{En le 4}
For $l=1,\dots,k-1$, we have
\begin{align}\label{energy 3}
\frac{d}{dt}\int
\na^{l}u\cdot \na\na^l\varrho+\norm{\na^l\varrho}_{L^2}^2+\norm{\na^{l+1}\varrho}_{L^2}^2\ls\norm{\na^{l+2}u}_{L^2}^2+\norm{\na u}_{L^2}^2+\norm{\varrho}_{L^2}^2.
\end{align}
\end{lemma}
\begin{proof}

Applying $\na^l$ to the second equation in $\eqref{NSP per}$, multiplying the resulting identity by $\na^l\na\varrho$ and then integrating over $\r3$, we obtain
\begin{align}\label{yi gaojie varrho1}
&\int\na^l\pa_tu\cdot\na^l\na\varrho+\int h'(\rho_s)|\na^{l+1}\varrho|^2-\int\na^l\na\Phi\cdot\na^l\na\varrho\nonumber\\
&\quad\ls \left( \norm{\left[\na^{l+1},h'(\rho_s)\right]\varrho}_{L^2}+\norm{\na^{l+2}u}_{L^2} +\norm{\na^l\left(u\cdot\na u\right)}_{L^2}\right.\nonumber\\
&\qquad\ \left. +\norm{\na^{l+1}\mathcal{R}}_{L^2}+\norm{\na^l\left(\left(\frac{1}{\varrho+\rho_s}-\frac{1}{\rho_s}\right)\na^2u\right)}_{L^2}\right)\norm{\na^{l+1}\varrho}_{L^2}
.\end{align}
As \eqref{san 0jie varrho}, we deduce
\begin{align}\label{yi gaojie varrho2}
&-\int  \na^{l} \partial_tu\cdot\na\na^l\varrho=-\frac{d}{dt}\int
\na^{l}u\cdot\na\na^l\varrho +\int\na^{l}
 u\cdot \na \na^{l}\pa_t\varrho\nonumber\\
&\quad=-\frac{d}{dt}\int
\na^{l}u\cdot\na\na^l\varrho+\int \na^{l} \diver  u\na^{l}\left(
u \cdot \na\varrho+ \varrho\diver  u+\rho_s\diver u+u\cdot\na \rho_s\right)\nonumber\\
&\quad\le-\frac{d}{dt}\int
\na^{l}u\cdot \na\na^l\varrho+\norm{\na^{l+1}u}_{L^2}^2+\norm{\na^{l}(u \cdot \na\varrho)}_{L^2}^2\nonumber\\
&\qquad+\norm{\na^{l}(\varrho\diver  u)}_{L^2}^2+\norm{\na^{l}(\rho_s\diver  u)}_{L^2}^2+\norm{\na^{l}(u\cdot\na\rho_s)}_{L^2}^2.
\end{align}
As \eqref{er 0jie varrho}, we have
\begin{align}\label{liu gaojie varrho}
-\int\na^l\na\Phi\cdot\na^l\na\varrho=\int\na^l\De\Phi\na^l\varrho=\norm{\na^l\varrho}_{L^2}^2.
\end{align}
Plugging \eqref{yi gaojie varrho2}--\eqref{liu gaojie varrho} into \eqref{yi gaojie varrho1}, and applying the product estimates \eqref{product estimate}, Lemma \ref{A1} and the nonlinear estimates \eqref{A31} as in Lemma \ref{En le 2}, we obtain
\begin{align}\label{yi gaojie varrho}
&\frac{d}{dt}\int
\na^{l}u\cdot\na\na^l\varrho+\int h'(\rho_s)|\na^{l+1}\varrho|^2+\norm{\na^l\varrho}_{L^2}^2\nonumber\\
&\quad\ls \norm{\left[\na^{l+1},h'(\rho_s)\right]\varrho}_{L^2}\norm{\na^{l+1}\varrho}_{L^2}\nonumber\\
&\qquad+\left(\norm{\na^{l+2}u}_{L^2} +\norm{ \na u }_{L^2}+\de\left(\norm{\na\varrho}_{L^2}+\norm{\na^{l+1}\varrho}_{L^2}\right)\right)\norm{\na^{l+1}\varrho}_{L^2}.
\end{align}
By the commutator estimates \eqref{commutator estimate}, we obtain
\begin{align}\label{er gaojie varrho}
\norm{\left[\na^{l+1},h'(\rho_s)\right]\varrho}_{L^2}&\ls\norm{\na h'(\rho_s)}_{L^\infty}\norm{\na^l\varrho}_{L^2}+\norm{\na^{l+1}h'(\rho_s)}_{L^3}\norm{\varrho}_{L^6}\nonumber
\\&\ls\norm{\na^l\varrho}_{L^2}+\norm{\na\varrho}_{L^2}.
\end{align}

Plugging \eqref{er gaojie varrho} into \eqref{yi gaojie varrho}, by Cauchy's inequality and since $\de$ is small, we have
\begin{align}\label{san gaojie varrho}
\frac{d}{dt}&\int\na^{l}u\cdot \na\na^l\varrho+\norm{\na^{l+1}\varrho}_{L^2}^2+\norm{\na^l\varrho}_{L^2}^2\nonumber\\
&\ls\norm{\na^{l+2}u}_{L^2}^2 +\norm{\na u}_{L^2}^2+\norm{\na^l\varrho}_{L^2}^2+\norm{\na\varrho}_{L^2}^2.
\end{align}
By the interpolation and Young's inequality, we further deduce \eqref{energy 3} from \eqref{san gaojie varrho}.
\end{proof}

%%%%%%%%%%%%%%%%%%%%%%%%%%%%%%%%%%%%%%%%%%%%%%%%%%%%%%%%%%%%%%%%%%%%%%%%%%%%%
\subsection{Proof of Theorem \ref{existence}}
%%%%%%%%%%%%%%%%%%%%%%%%%%%%%%%%%%%%%%%%%%%%%%%%%%%%%%%%%%%%%%%%%%%%%%%%%%%%%
In this subsection, we will prove Theorem \ref{existence}.
 Multiplying the estimates \eqref{energy 3} of Lemma \ref{En le 4} with $l=k-1$ by a small fixed constant $\epsilon_1$, and then adding it to the estimates \eqref{energy 2} of Lemma \ref{En le 3}, by the interpolation and Young's inequality, we deduce
\begin{align}\label{new energy 1}
\frac{d}{dt}\left(\int u\cdot\na(h'(\rho_s)\varrho)+\epsilon_1\int\na^{k-1}u\cdot \na^k\varrho\right)+\norm{ \varrho }_{H^k}^2\ls\norm{\na^{k+1} u }_{L^2}^2+\norm{\na u }_{L^2}^2.
\end{align}
Multiplying \eqref{new energy 1} by a small fixed constant $\epsilon_2$ and then adding it to the estimates \eqref{energy 1} of Lemma \ref{En le 2} with $l=k$, we obtain
\begin{align}\label{new energy 11}
&\frac{d}{dt}\left(\int h'(\rho_s)|\na^k\varrho|^2+\rho_s|\na^ku|^2+\eps_2\left(\int u\cdot\na(h'(\rho_s)\varrho)+\epsilon_1\int\na^{k-1}u\cdot \na^k\varrho\right)\right)\nonumber\\
&\quad+\norm{\na^{k+1}u}_{L^2}^2+\norm{ \varrho }_{H^k}^2\ls (\va+\delta)\norm{ \varrho }_{H^k}^2 +C_\varepsilon\left( \norm{ \na^{k}u }_{L^2}^2+ \norm{\na u }_{L^2}^2\right ).
\end{align}
Taking $\va$ sufficiently small in \eqref{new energy 11} and since $\delta$ is small, by the interpolation and Young's inequality, we then have
\begin{align}\label{new energy 2}
&\frac{d}{dt}\left(\int h'(\rho_s)|\na^k\varrho|^2+\rho_s|\na^ku|^2+\eps_2\left(\int u\cdot\na(h'(\rho_s)\varrho)+\epsilon_1\int\na^{k-1}u\cdot \na^k\varrho\right)\right)\nonumber\\
&\quad+\norm{\na u}_{H^k}^{2}+\norm{ \varrho }_{H^k}^2 \ls \norm{\na u }_{L^2}^2.
\end{align}
Multiplying \eqref{new energy 2} by a small fixed constant $\epsilon_3$ and then adding it to the estimates \eqref{energy 0} of Lemma \ref{En le 1}, since $\de$ is small, we deduce
\begin{align}\label{new energy 3}
&\frac{d}{dt}\left(\int\left(h'(\rho_s)\varrho^2+\rho_s|u|^2+|\na\Phi|^2\right)+\eps_3\left(\int h'(\rho_s)|\na^k\varrho|^2+\rho_s|\na^ku|^2\right.\right.\nonumber
\\&\quad\left.\left.+\eps_2\left(\int u\cdot\na(h'(\rho_s)\varrho)+\epsilon_1\int\na^{k-1}u\cdot \na^k\varrho\right)\right)\right) +\norm{\na u}_{H^k}^{2}+\norm{ \varrho }_{H^k}^2 \le 0.
\end{align}
Note that the expression under the time differentiation in \eqref{new energy 3} with properly small $\eps_1$, $\eps_2$ and $\eps_3$ is equivalent to $\norm{ \varrho }_{H^k}^{2}+\norm{  u  }_{H^k}^{2} +\norm{ \na \Phi }_{L^2}^{2}$.
Hence, integrating \eqref{new energy 3} directly in time, we obtain \eqref{energy inequlity}. By a standard continuity argument, we then close the a priori estimates \eqref{a priori} if we assume at initial time that $\norm{ \varrho_0 }_{H^k}+\norm{ u_0 }_{H^k}+\norm{\nabla\Phi_0}_{L^2}$ is sufficiently small. The global solution then follows by a standard continuity argument combined with the local existence of solutions. The proof of Theorem \ref{existence} is thus completed.\hfill$\Box$

%%%%%%%%%%%%%%%%%%%%%%%%%%%%%%%%%%%%%%%%%%%%%%%%%%%%%%%%%%%%%%%%%%%%%%%%%%%%%%
\section{Time decay with small doping profile}\label{section4}
%%%%%%%%%%%%%%%%%%%%%%%%%%%%%%%%%%%%%%%%%%%%%%%%%%%%%%%%%%%%%%%%%%%%%%%%%%%%%%
In this section, we will derive the time decay rates of the solution to \eqref{NSP} towards the steady state. For this, we need to require that the doping profile is of small variation, that is, $b(x)$ is near the constant $\bar b$. Then according to Proposition \ref{prop}, $\rho_s(x)$ is near $\bar\rho$ with $\bar\rho=\bar b$. Owing to this fact, we may rewrite \eqref{NSP per} as:
\begin{align} \label{NSP per2}
\begin{cases}
\displaystyle\pa_t\varrho+\bar\rho \diver u=-\diver((\varrho+\rho_s-\bar\rho)  u),   \\
\displaystyle\pa_tu+h'(\bar\rho)\na\varrho-\frac{1}{\bar\rho}\left(\mu\Delta u+(\mu+\mu')\na\diver u\right)-\na\Phi\\
\displaystyle\ \ =\left(\frac{1}{\varrho+\rho_s}-\frac{1}{\bar\rho}\right)\left(\mu\Delta u+(\mu+\mu')\na\diver u\right)-u\cdot\na u-\na \mathcal{R}-\na((h'(\rho_s)-h'(\bar\rho))\varrho), \\
\Delta\Phi=\varrho, \\
(\varrho,u)\mid_{t=0}=(\varrho_0, u_0).
\end{cases}
\end{align}

%%%%%%%%%%%%%%%%%%%%%%%%%%%%%%%%%%%%%%%%%%%%%%%%%%%%%%%%%%%%%%%%%%%%%%%%%%%%%
\subsection{Energy estimates}
%%%%%%%%%%%%%%%%%%%%%%%%%%%%%%%%%%%%%%%%%%%%%%%%%%%%%%%%%%%%%%%%%%%%%%%%%%%%%

In this subsection, under the assumptions of Theorem \ref{decay}, we will derive the further energy estimates for the global solutions to the Navier-Stokes-Poisson equations \eqref{NSP per2} obtained in Theorem \ref{existence}. By Theorem \ref{existence}, the assumptions of Theorem \ref{decay} and Proposition \ref{prop}, we have that
\begin{align}\label{a priori2}
\norm{(\varrho,u)(t)}_{H^k}+\norm{\na\Phi(t)}_{L^2} +\norm{\rho_s-\bar\rho}_{H^{k+1}}\le \de
\end{align}
for some small constant $\de>0$.

We first derive the energy estimates for the derivatives of the solution of order $l$ with $l\in [0,k]$. Note that now $l$ is not required to be an integer.
\begin{lemma}\label{En le 1decay}
For $0\le l\le k$, we have
\begin{align}\label{energy 1decay}
&\frac{d}{dt}\int\left(h'(\bar\rho)|\na^l\varrho|^2+\bar\rho|\na^lu|^2+|\na^l\na\Phi|^2\right)+\norm{\na^{l+1}u}_{L^2}^2\nonumber\\
&\quad \ls \de\left( \norm{ \na^{l}\varrho }_{L^2}^2+\norm{\na^l\na\Phi }_{L^2}^2+\norm{ \na^{l}u }_{L^2}^2+\norm{\varrho}_{L^\infty}^2
+\norm{u}_{L^\infty}^2+\norm{\na u}_{L^3}^2+\norm{\na^2u}_{L^3}^2\right).
\end{align}
\end{lemma}
\begin{proof}
Applying $\na^l$ to the first two equations in \eqref{NSP per2} and then multiplying the resulting identities by $h'(\bar\rho)\na^l\varrho$ and $\bar\rho\na^lu$ respectively, summing up them and then integrating over $\r3$, we obtain
\begin{align} \label{yi u2}
&\frac{1}{2}\frac{d}{dt}\int\left(h'(\bar\rho)|\na^l\varrho|^2+\bar\rho|\na^lu|^2\right)+ \int \mu|\na^{l+1}u|^2+(\mu+\mu')|\na^l\diver u|^2- \int \bar\rho \na^lu\cdot\na^l\na\Phi\nonumber\\
&\quad=-\int h'(\bar\rho)\na^l\varrho\na^l\diver((\varrho+\rho_s-\bar\rho) u)-\int\bar\rho\na^lu\cdot\na^l(u\cdot\na u)\nonumber\\
&\qquad-\int\bar\rho\na^lu\cdot\na^l\na\mathcal{R}-\int \bar\rho\na^lu\cdot\na^l \na((h'(\rho_s)-h'(\bar\rho))\varrho) \nonumber\\
&\qquad+\int\bar\rho\na^lu\cdot \na^l\left(\left(\frac{1}{\varrho+\rho_s}-\frac{1}{\bar\rho}\right)\left(\mu\Delta u+(\mu+\mu')\na\diver u\right)\right).
\end{align}

First, we estimate the terms on the right-hand side of \eqref{yi u2}. By the commutator notation \eqref{commutator}, we obtain
\begin{align}\label{yi youyi2}
&-\int h'(\bar\rho)\na^l\varrho\na^l\diver((\varrho+\rho_s-\bar\rho) u)\nonumber\\
&\quad=-\int h'(\bar\rho)\na^l\varrho\na^l\left(u\cdot\na(\varrho+\rho_s-\bar\rho)+(\varrho+\rho_s-\bar\rho)\diver u\right)\nonumber\\
&\quad=-\int h'(\bar\rho)u\cdot\na\na^l\varrho\na^l\varrho-\int h'(\bar\rho)\na^l\varrho\left[\na^l,u\right]\cdot\na\varrho\nonumber\\
&\qquad-\int h'(\bar\rho)\na^l\varrho\na^l\left(u\cdot\na(\rho_s-\bar\rho)\right)-\int h'(\bar\rho)\na^l\varrho\na^l\left((\varrho+\rho_s-\bar\rho)\diver u\right).
\end{align}
By the integration by parts, we obtain
\begin{align}\label{er youyi2}
&-\int h'(\bar\rho)u\cdot\na\na^l\varrho\na^l\varrho=-\frac12\int h'(\bar\rho)u\cdot\na|\na^l\varrho|^2=\frac12\int h'(\bar\rho)\diver u|\na^l\varrho|^2\nonumber\\
&\quad\ls\norm{\diver u}_{L^\infty}\norm{\na^l\varrho}_{L^2}^2\ls\de\norm{\na^l\varrho}_{L^2}^2.
\end{align}
By the commutator estimates \eqref{commutator estimate}, we obtain
\begin{align}\label{san youyi2}
&-\int h'(\bar\rho)\na^l\varrho\left[\na^l,u\right]\cdot\na\varrho\ls\norm{\na^l\varrho}_{L^2}\norm{\left[\na^l,u\right]\cdot\na\varrho}_{L^2}\nonumber\\
&\quad\ls\norm{\na^l\varrho}_{L^2}\left(\norm{\na u}_{L^\infty}\norm{\na^l\varrho}_{L^2}+\norm{\na^lu}_{L^6}\norm{\na\varrho}_{L^3}\right)\nonumber\\
&\quad\ls\de\left(\norm{\na^l\varrho}_{L^2}^2+\norm{\na^{l+1}u}_{L^2}^2\right).
\end{align}
By the product estimates  \eqref{product estimate} and \eqref{a priori2}, we obtain
\begin{align}\label{si youyi21}
&-\int h'(\bar\rho)\na^l\varrho\na^l\left(u\cdot\na(\rho_s-\bar\rho)\right)\ls\norm{\na^l\varrho}_{L^2}\norm{\na^l\left(u\cdot\na(\rho_s-\bar\rho)\right)}_{L^2}\nonumber\\
&\quad\ls\norm{\na^l\varrho}_{L^2}\left(\norm{u}_{L^\infty}\norm{\na^{l+1}(\rho_s-\bar\rho)}_{L^2}+\norm{\na^lu}_{L^6}\norm{\na(\rho_s-\bar\rho)}_{L^3}\right)\nonumber\\
&\quad\ls\de\left(\norm{\na^l\varrho}_{L^2}^2+\norm{\na^{l+1}u}_{L^2}^2+\norm{u}_{L^\infty}^2\right)
\end{align}
and
\begin{align}\label{si youyi22}
&-\int h'(\bar\rho)\na^l\varrho\na^l\left((\varrho+\rho_s-\bar\rho)\diver u\right)\ls\norm{\na^l\varrho}_{L^2}\norm{\na^l\left((\varrho+\rho_s-\bar\rho)\diver u\right)}_{L^2}\nonumber\\
&\quad\ls\norm{\na^l\varrho}_{L^2}\left(\norm{\varrho+\rho_s-\bar\rho}_{L^\infty}\norm{\na^l\diver u}_{L^2}+\norm{\na^l\varrho}_{L^2}\norm{\na u}_{L^\infty}+\norm{\na^l(\rho_s-\bar\rho)}_{L^6}\norm{\na u}_{L^3}\right)\nonumber\\
&\quad\ls\de\left(\norm{\na^l\varrho}_{L^2}^2+\norm{\na^{l+1}u}_{L^2}^2+\norm{\na u}_{L^3}^2\right).
\end{align}
Hence, we may conclude that
\begin{align}\label{youyi decay}
-\int h'(\bar\rho)\na^l\varrho\na^l\diver((\varrho+\rho_s-\bar\rho) u)\ls\de\left(\norm{\na^l\varrho}_{L^2}^2+\norm{\na^{l+1}u}_{L^2}^2+\norm{u}_{L^\infty}^2+\norm{\na u}_{L^3}^2\right).
\end{align}

Next, by H\"older's and Sobolev's inequalities, the product estimates \eqref{product estimate} and the interpolation estimates \eqref{A.1}, we obtain that for $l=0$,
\begin{align}\label{yi youer decay}
-\int\bar\rho\na^{l} u\cdot\na^{l}\left(u\cdot\na u \right)=-\int\bar\rho u\cdot\left(u\cdot\na u \right)\ls\norm{u}_{L^3}\norm{u}_{L^6}\norm{\na u}_{L^2}\ls\de\norm{\na u}_{L^2}^2;
\end{align}
for $l\ge1$,
\begin{align}\label{er youer decay}
-&\int\bar\rho\na^{l} u\cdot\na^{l}\left(u\cdot\na u \right)
=\norm{\na^lu}_{L^6}\norm{\na^l\left(u\cdot\na u \right)}_{L^{6/5}}\nonumber\\
&\ls \norm{\na^{l+1}u}_{L^2}\left(\norm{u}_{L^3}\norm{\na^{l+1}u}_{L^2}+\norm{\na^{l}u}_{L^2}\norm{\na u}_{L^3}\right)\nonumber\\
&\ls \norm{\na^{l+1}u}_{L^2}\left(\norm{u}_{L^3}\norm{\na^{l+1}u}_{L^2}+\norm{u}_{L^2}^{\frac{1}{l+1}}\norm{\na^{l+1}u}_{L^2}^{\frac{l}{l+1}}\norm{\na^{\frac{l+1}{2l}} u}_{L^2}^{\frac{l}{l+1}}\norm{\na^{l+1}u}_{L^2}^{\frac{1}{l+1}}\right)\nonumber\\
&\ls\de\norm{\na^{l+1} u}_{L^2}^2.
\end{align}
Hence, we may conclude that for $l\ge0$,
\begin{align}\label{youer decay}
-\int\bar\rho\na^{l} u\cdot\na^{l}\left(u\cdot\na u \right)\ls\de\norm{\na^{l+1} u}_{L^2}^2.
\end{align}
Integrating by parts and by the product estimates \eqref{product estimate}, the nonlinear estimates \eqref{A32} and \eqref{a priori2}, we obtain
\begin{align}\label{yousan decay}
&-\int\bar\rho\na^lu\cdot\na^l\na\mathcal{R}=\int\bar\rho\na^{l}{\rm div }u \na^{l}\mathcal{R}\ls\norm{\na^{l+1}u}_{L^2}\norm{\na^{l}\mathcal{R}}_{L^2}\nonumber\\
&\quad\ls\de\left(\norm{\varrho}_{L^\infty}^2+\norm{\na^l\varrho}_{L^2}^2+\norm{\na^{l+1}u}_{L^2}^2\right)
\end{align}
and
\begin{align} \label{yousi decay}
& -\int \bar\rho\na^lu\cdot\na^l \na((h'(\rho_s)-h'(\bar\rho))\varrho)=\int \bar\rho\na^{l+1}u\cdot\na^l((h'(\rho_s)-h'(\bar\rho))\varrho) \nonumber\\
&\quad \ls \norm{\na^{l+1}u}_{L^2}\norm{\na^l  ((h'(\rho_s)-h'(\bar\rho))\varrho)}_{L^2}\nonumber\\
&\quad \ls \norm{\na^{l+1}u}_{L^2}\left(\norm{h'(\rho_s)-h'(\bar\rho)}_{L^\infty}\norm{\na^l \varrho}_{L^2}+\norm{\na^l  (h'(\rho_s)-h'(\bar\rho))}_{L^2}\norm{\varrho}_{L^\infty}\right)\nonumber\\
&\quad \ls \delta \left(\norm{\na^{l+1}u}_{L^2}^2+\norm{\na^l \varrho}_{L^2}^2+ \norm{\varrho}_{L^\infty}^2\right).
\end{align}

We now estimate the last term on the right-hand side of \eqref{yi u2}. For $l=0$, we easily obtain
\begin{align}\label{yi youwu decay}
\int&\bar\rho u\cdot\left(\frac{1}{\varrho+\rho_s}-\frac{1}{\bar\rho}\right)\left(\mu\Delta u+(\mu+\mu')\na\diver u\right)\nonumber\\
&\ls\norm{u}_{L^6}\norm{\frac{1}{\varrho+\rho_s}-\frac{1}{\bar\rho}}_{L^2}\norm{\na^2u}_{L^3}\ls\de\left(\norm{\na u}_{L^2}^2+\norm{\na^2u}_{L^3}^2\right).
\end{align}
For $l\ge1$, by the integration by parts, the product estimates \eqref{product estimate} and \eqref{a priori2}, we obtain
\begin{align}\label{er youwu decay}
\int&\bar\rho\na^lu\cdot \na^l\left(\left(\frac{1}{\varrho+\rho_s}-\frac{1}{\bar\rho}\right)\left(\mu\Delta u+(\mu+\mu')\na\diver u\right)\right)\nonumber\\
&=-\int\bar\rho\na^{l+1}u\cdot \na^{l-1}\left(\left(\frac{1}{\varrho+\rho_s}-\frac{1}{\bar\rho}\right)\left(\mu\Delta u+(\mu+\mu')\na\diver u\right)\right)
\nonumber\\
&\ls\norm{\na^{l+1}u}_{L^2}\norm{\na^{l-1}\left(\left(\frac{1}{\varrho+\rho_s}-\frac{1}{\bar\rho}\right)\left(\mu\Delta u+(\mu+\mu')\na\diver u\right)\right)}_{L^2}\nonumber\\
&\ls \norm{\na^{l+1}u}_{L^2}\left(\norm{ \frac{1}{\varrho+\rho_s}-\frac{1}{\bar\rho} }_{L^\infty}\norm{\na^{l+1}u}_{L^2}+\norm{\na^{l-1} \left(\frac{1}{\varrho+\rho_s}-\frac{1}{\bar\rho}\right)}_{L^6}\norm{\na^2u}_{L^3}\right)\nonumber\\
&\ls \delta \left(\norm{\na^{l+1}u}_{L^2}^2+ \norm{\na^2u}_{L^3}^2\right).
\end{align}
Hence, we may conclude that for $l\ge0$,
\begin{align}\label{youwu decay}
\int&\bar\rho\na^lu\cdot \na^l\left(\left(\frac{1}{\varrho+\rho_s}-\frac{1}{\bar\rho}\right)\left(\mu\Delta u+(\mu+\mu')\na\diver u\right)\right)\nonumber\\
&\ls \delta \left(\norm{\na^{l+1}u}_{L^2}^2+ \norm{\na^2u}_{L^3}^2\right).
\end{align}

Now, we turn to estimate the terms on the left-hand side of \eqref{yi u2}. For the second term, we deduce from \eqref{viscosity} that
\begin{align}\label{zuoer decay}
\int\mu|\na^{l+1}u|^2+(\mu+\mu')|\na^l\diver u|^2\gt\norm{\na^{l+1}u}_{L^2}^2.
\end{align}
For the remaining Poisson term, we integrate by parts and use the first equation and the Poisson equation in $\eqref{NSP per2}$ to obtain

\begin{align}\label{yi zuosan decay}
&- \int \bar\rho\na^lu\cdot\na^l\na\Phi=\int\na^l\Phi\,\na^l\diver(\bar\rho u)\nonumber\\
&\quad=- \int\na^l\Phi\na^l\pa_t\varrho+ \na^l\Phi\na^l\diver((\varrho+\rho_s-\bar\rho)  u)\nonumber\\
&\quad=- \int\na^l\Phi\na^l\pa_t\De\Phi-\na^l((\varrho+\rho_s-\bar\rho)  u)\cdot\na^l\na\Phi\nonumber\\
&\quad= \frac{1}{2}\frac{d}{dt}\int|\na^l\na\Phi|^2+\int\na^l((\varrho+\rho_s-\bar\rho)  u)\cdot\na^l\na\Phi.
\end{align}
By the product estimates \eqref{product estimate} and \eqref{a priori2}, we obtain
\begin{align}\label{er zuosan decay}
&\int\na^l((\rho_s-\bar\rho) u)\cdot\na^l\na\Phi
\ls\norm{\na^l((\rho_s-\bar\rho) u)}_{L^{2}}\norm{\na^l\na\Phi}_{L^2}\nonumber\\
&\quad\ls \left(\norm{\rho_s-\bar\rho}_{L^\infty}\norm{\na^{l}u}_{L^2}+\norm{\na^{l}(\rho_s-\bar\rho)}_{L^2}\norm{ u}_{L^\infty}\right)\norm{\na^l\na\Phi}_{L^2}\nonumber\\
&\quad\ls\de\left(\norm{\na^l\na\Phi}_{L^2}^2+\norm{\na^{l}u}_{L^2}^2+\norm{ u}_{L^\infty}^2\right)
\end{align}
and
\begin{align}\label{san zuosan decay}
&\int\na^l( \varrho  u)\cdot\na^l\na\Phi
\ls\norm{\na^l(\varrho u)}_{L^{2}}\norm{\na^l\na\Phi}_{L^2}\nonumber\\
&\quad\ls \left(\norm{\varrho}_{L^\infty}\norm{\na^{l}u}_{L^2}+\norm{\na^{l}\varrho}_{L^2}\norm{ u}_{L^\infty}\right)\norm{\na^l\na\Phi}_{L^2}\nonumber\\
&\quad\ls\de\left(\norm{\na^l\na\Phi}_{L^2}^2+\norm{\na^{l}u}_{L^2}^2+\norm{\na^{l}\varrho}_{L^2}^2\right).
\end{align}
Hence, we may conclude that
\begin{align}\label{zuosan decay}
-\int \bar\rho\na^lu\cdot\na^l\na\Phi\ge \frac{1}{2}\frac{d}{dt}\int|\na^l\na\Phi|^2 -C\de\left(\norm{\na^l\na\Phi}_{L^2}^2+\norm{\na^{l}\varrho}_{L^2}^2+\norm{\na^{l}u}_{L^2}^2
+\norm{u}_{L^\infty}^2\right).
\end{align}

Consequently, plugging the estimates \eqref{youyi decay}, \eqref{youer decay}--\eqref{yousi decay}, \eqref{youwu decay}--\eqref{zuoer decay} and \eqref{zuosan decay} into \eqref{yi u2}, since $\de$ is small, we deduce \eqref{energy 1decay}.
\end{proof}

We now recover the dissipation estimates for $\varrho$.
\begin{lemma} \label{En le 2decay}
For $0\le l\le k-1$, we have
\begin{align}\label{energy 2decay}
\frac{d}{dt}&\int
\na^{l}u\cdot \na\na^l\varrho+\norm{\na^l\varrho}_{L^2}^2+\norm{\na^{l+1}\varrho}_{L^2}^2\nonumber\\
&\ls\norm{\na^{l+1}u}_{L^2}^2+\norm{\na^{l+2}u}_{L^2}^2+\de\left(\norm{\varrho}_{L^\infty}^2+\norm{u}_{L^\infty}^2+\norm{\na u}_{L^3}^2+\norm{\na^2u}_{L^3}^2\right).
\end{align}
\end{lemma}
\begin{proof}
Applying $\na^l$ to the second equation in $\eqref{NSP per2}$, multiplying the resulting identity by $\na^l\na\varrho$ and then integrating over $\r3$, we obtain
\begin{align}
&\int\na^l\pa_tu\cdot\na^l\na\varrho+\int h'(\bar\rho)|\na^{l+1}\varrho|^2-\int\na^l\na\Phi\cdot\na^l\na\varrho\nonumber\\
&\quad\ls \left( \norm{\na^{l+1}\left((h'(\rho_s)-h'(\bar\rho))\varrho\right)}_{L^2}+\norm{\na^{l+2}u}_{L^2} +\norm{\na^l\left(u\cdot\na u\right)}_{L^2}\right.\nonumber\\
&\qquad\ \left. +\norm{\na^{l+1}\mathcal{R}}_{L^2}+\norm{\na^l\left(\left(\frac{1}{\varrho+\rho_s}-\frac{1}{\bar\rho}\right)\na^2u\right)}_{L^2}\right)\norm{\na^{l+1}\varrho}_{L^2}
.\end{align}
As \eqref{san 0jie varrho}, we deduce
\begin{align}
&-\int  \na^{l} \partial_tu\cdot\na\na^l\varrho=-\frac{d}{dt}\int
\na^{l}u\cdot\na\na^l\varrho +\int\na^{l}
 u\cdot \na \na^{l}\pa_t\varrho\nonumber\\
&\quad=-\frac{d}{dt}\int
\na^{l}u\cdot\na\na^l\varrho+\int \na^{l} \diver  u\na^{l}\left(
u \cdot \na\varrho+ \varrho\diver  u+\rho_s\diver u+u\cdot\na \rho_s\right)\nonumber\\
&\quad\le-\frac{d}{dt}\int
\na^{l}u\cdot \na\na^l\varrho+\norm{\na^{l+1}u}_{L^2}^2+\norm{\na^{l}(u \cdot \na\varrho)}_{L^2}^2\nonumber\\
&\qquad+\norm{\na^{l}(\varrho\diver  u)}_{L^2}^2+\norm{\na^{l}(\rho_s\diver  u)}_{L^2}^2+\norm{\na^{l}(u\cdot\na\rho_s)}_{L^2}^2.
\end{align}
As \eqref{yi gaojie varrho2}--\eqref{yi gaojie varrho}, applying the product estimates \eqref{product estimate} and the nonlinear estimates \eqref{A32} as in Lemma \ref{En le 1decay}, we obtain
\begin{align}\label{yi gaojie varrhodecay}
&\frac{d}{dt}\int
\na^{l}u\cdot\na\na^l\varrho+\int h'(\rho_s)|\na^{l+1}\varrho|^2+\norm{\na^l\varrho}_{L^2}^2\nonumber\\
&\quad\ls \left(\norm{\na^{l+1}u}_{L^2} +\norm{\na^{l+2}u}_{L^2} +\de\norm{\varrho}_{L^\infty}+\de\norm{u}_{L^\infty}\right.\nonumber\\
&\qquad\left.+\de\norm{\na u}_{L^3}+\de\norm{\na^2u }_{L^3}+\de\norm{\na^{l+1}\varrho}_{L^2}\right)\norm{\na^{l+1}\varrho}_{L^2}.
\end{align}
By Cauchy's inequality and since $\de$ is small, we deduce \eqref{energy 2decay} from \eqref{yi gaojie varrhodecay}.
\end{proof}

We now combine Lemmas \ref{En le 1decay}--\ref{En le 2decay} to derive the following proposition.
\begin{Proposition}\label{proposition decay}
Let $k\ge3$ and $0\le\ell\le 3/2$. Then there exists an energy functional $\mathcal{E}_\ell^k$
equivalently to $\norm{\na^{\ell}(\varrho,u,\na\Phi)}_{H^{k-\ell}}^2$
such that
\begin{align}\label{energy 5decay}
\frac{d}{dt}\mathcal{E}_\ell^k+\norm{\na^{\ell}\varrho}_{H^{k-\ell}}^2+\norm{\na^{\ell+1}u}_{H^{k-\ell}}^2
\ls \de\left( \norm{ \na^{\ell}\na\Phi }_{L^2}^2+\norm{ \na^{\ell}u }_{L^2}^2+\norm{\varrho}_{L^\infty}^2
+\norm{u}_{L^\infty}^2\right).
\end{align}
\end{Proposition}
\begin{proof}
Summing up the estimates \eqref{energy 1decay} of Lemma \ref{En le 1decay} for from $l=\ell$ to $k$, by the Poisson equation in \eqref{NSP per2}, we obtain
\begin{align}\label{1proof1}
&\frac{d}{dt}\left( h'(\bar\rho)\norm{\na^{\ell}
\varrho}_{H^{k-\ell}}^2+\bar\rho\norm{\na^{\ell}u}_{H^{k-\ell}}^2+\norm{\na^\ell\na\Phi}_{H^{k-\ell}}^2\right) + \norm{\na^{\ell+1} u}_{H^{k-\ell}}^2\nonumber\\
&\quad\ls \de\left(\norm{\na^{\ell}\varrho}_{H^{k-\ell}}^2+\norm{\na^{\ell}\na\Phi}_{L^2}^2+\norm{\na^{\ell}u}_{H^{k-\ell}}^2+\norm{\varrho}_{L^\infty}^2
+\norm{u}_{L^\infty}^2+\norm{\na u}_{L^3}^2+\norm{\na^2u}_{L^3}^2\right).
\end{align}
Summing up the estimates \eqref{energy 2decay} of Lemma \ref{En le 2decay} for from $l=\ell$ to $k-1$, we obtain
\begin{align}\label{1proof2}
&\frac{d}{dt}\sum_{\ell\le l\le k-1}\int \na^lu\cdot\na\na^l\varrho + \norm{\na^{\ell}\varrho}_{H^{k-\ell}}^2\nonumber\\
&\quad\ls \norm{\na^{\ell+1}u}_{H^{k-\ell}}^2+\de\left(\norm{\varrho}_{L^\infty}^2+\norm{u}_{L^\infty}^2+\norm{\na u}_{L^3}^2+\norm{\na^2u}_{L^3}^2\right).
\end{align}
Multiplying \eqref{1proof2} by a small constant $\eps>0$ and then adding the resulting inequality to \eqref{1proof1}, since $\delta$ is small, we deduce that for $0\le \ell\le k-1$,
\begin{align}\label{1proof3}
\frac{d}{dt}&\left(h'(\bar\rho)\norm{\na^{\ell} \varrho}_{H^{k-\ell}}^2+\bar\rho\norm{\na^{\ell}
u}_{H^{k-\ell}}^2+ \norm{\na^\ell\na\Phi}_{H^{k-\ell}}^2+\epsilon\sum_{\ell\le l\le k-1}\int
\na^lu\cdot\na\na^l\varrho\right)\nonumber\\
&+ \norm{\na^{\ell}\varrho}_{H^{k-\ell}}^2+\norm{\na^{\ell+1}
u}_{H^{k-\ell}}^2\nonumber\\
&\ls\de\left(\norm{ \na^{\ell}\na\Phi }_{L^2}^2+\norm{ \na^{\ell}u }_{L^2}^2+\norm{\varrho}_{L^\infty}^2+\norm{u}_{L^\infty}^2+\norm{\na u}_{L^3}^2+\norm{\na^2u}_{L^3}^2\right).
\end{align}
We define $\mathcal{E}_\ell^k$ to be the expression under the time derivative in \eqref{1proof3}. Since
$\eps$ is small, $\mathcal{E}_\ell^k$ is equivalent to $\norm{\na^{\ell}(\varrho,u,\na\Phi)}_{H^{k-\ell}}^2$.
Then we deduce that for $0\le \ell\le k-1$,
\begin{align}\label{1proof5}
\frac{d}{dt}&\mathcal{E}_\ell^k+\norm{\na^{\ell}\varrho}_{H^{k-\ell}}^2+\norm{\na^{\ell+1}u}_{H^{k-\ell}}^2\nonumber\\
&\ls\de\left(\norm{ \na^{\ell}\na\Phi }_{L^2}^2+\norm{ \na^{\ell}u }_{L^2}^2+\norm{\varrho}_{L^\infty}^2+\norm{u}_{L^\infty}^2+\norm{\na u}_{L^3}^2+\norm{\na^2u}_{L^3}^2\right).
\end{align}

Now we take $0\le \ell\le3/2$. Then by the interpolation, we have
\begin{align}\label{1proof7}
 \norm{\na u}_{L^3}^2\ls\norm{\na^{3/2} u}_{L^2}^2\ls \norm{\na^{\ell} u}_{L^2}^2+\norm{\na^{k} u}_{L^2}^2
 \end{align}
 and
\begin{align}\label{1proof8}
\norm{\na^2u}_{L^3}^2\ls\norm{\na^{5/2} u}_{L^2}^2\ls \norm{\na^{\ell+1} u}_{L^2}^2+\norm{\na^{k} u}_{L^2}^2.
\end{align}
Since $\de$ is small, \eqref{1proof5} implies  \eqref{energy 5decay} by \eqref{1proof7}--\eqref{1proof8}.
\end{proof}

%%%%%%%%%%%%%%%%%%%%%%%%%%%%%%%%%%%%%%%%%%%%%%%%%%%%%%%%%%%%%%%%%%%%%%%%%%%%%
\subsection{Duhamel form analysis}
%%%%%%%%%%%%%%%%%%%%%%%%%%%%%%%%%%%%%%%%%%%%%%%%%%%%%%%%%%%%%%%%%%%%%%%%%%%%%

In order to use the linear decay estimates for the linear system with constant coefficients, we will rewrite the Navier-Stokes-Poisson system \eqref{NSP per2} as the Navier-Stokes equations with a non-local self-consistent force in the following form:
\begin{align} \label{NSP per3}
\begin{cases}
\displaystyle\pa_t\varrho+\bar\rho \diver u=N^1,   \\
\displaystyle\pa_tu+h'(\bar\rho)\na\varrho-\frac{1}{\bar\rho}\left(\mu\Delta u+(\mu+\mu')\na\diver u\right)-\na\De^{-1}\varrho=N^2,\\
(\varrho,u)\mid_{t=0}=(\varrho_0, u_0),
\end{cases}
\end{align}
where the ``nonlinear" terms are given by
\begin{align}
 N^1=-\diver((\varrho+\rho_s-\bar\rho)  u)
 \end{align}
and
\begin{align}
N^2=-u\cdot\na u-\na \mathcal{R}-\na((h'(\rho_s)-h'(\bar\rho))\varrho) +\left(\frac{1}{\varrho+\rho_s}-\frac{1}{\bar\rho}\right)\left(\mu\Delta u+(\mu+\mu')\na\diver u\right).
\end{align}
By the Duhamel principle, the solution $(\varrho,u)$ to the problem \eqref{NSP per3} can be expressed as
\begin{align}\label{Duhamel}
(\varrho,u)(t)=e^{-t\mathbb{A}}(\varrho_0,u_0)+\int_0^te^{-(t-\tau)\mathbb{A}}(N^1,N^2)(\tau)\,d\tau.
\end{align}
Here the matrix differential operator $\mathbb{A}$ is defined by
\begin{align} \mathbb{A}=\left(
\begin{array}{cc}
0&\bar\rho\diver\\
h'(\bar\rho)\na-\na\De^{-1}&\quad-\frac{1}{\bar\rho}\left(\mu\De+(\mu+\mu')\na\diver\right)
\end{array}
\right).
\end{align}

In light of the analysis in \cite{LMZ} and \cite{W}, we have the followings about the time decay rates of the solution semigroup $e^{-t\mathbb{A}}$ of the linearized system of \eqref{NSP per3}.
\begin{lemma}\label{llll}
Let $(\widetilde{\varrho},\widetilde{u})=e^{-t\mathbb{A}}(\varrho_0,u_0)$. Then for $1\le p\le2$, $ q\ge2$ and $\ell\ge 0$, we have
\begin{align}\label{linear decay1}
\norm{\na^{\ell }\widetilde{\varrho}}_{L^q}
\ls(1+t)^{-\frac{3}{2}\left(\frac{1}{p}-\frac{1}{q}\right)-\frac{\ell}{2}-\frac{1}{2}}\left(\norm{(\na^{-1}\varrho_0,u_0)}_{L^p}+ \norm{\na^{\ell }(\varrho_0,u_0)}_{L^q}\right)
\end{align}
and
\begin{align}\label{linear decay2}
\norm{\na^\ell\widetilde{u}}_{L^q}
\ls(1+t)^{-\frac{3}{2}\left(\frac{1}{p}-\frac{1}{q}\right)-\frac{\ell}{2}}\left(\norm{(\na^{-1}\varrho_0,u_0)}_{L^p}+ \norm{\na^\ell(\varrho_0,u_0)}_{L^q}\right).
\end{align}
\end{lemma}

Applying Lemma \ref{llll} to \eqref{Duhamel}, we obtain the following proposition.
\begin{Proposition}\label{mingti wu}
It holds that for $1\le p,r\le 2$, $q\ge2$ and $\ell\ge0$,
\begin{align}\label{varrho l-1}
\norm{\na^{\ell}\varrho (t)}_{L^q}&\ls(1+t)^{-\frac{3}{2}\left(\frac{1}{p}-\frac{1}{q}\right)-\frac{\ell}{2}-\frac{1}{2}}\left(\norm{(\na^{-1}\varrho_0,u_0)}_{L^p}+ \norm{\na^{\ell}(\varrho_0,u_0)}_{L^q}\right) \nonumber\\
&\quad+\int_0^t(1+t-\tau)^{-\frac{3}{2}\left(\frac{1}{r}-\frac{1}{q}\right)-\frac{\ell}{2}-\frac{1}{2}}\left(\norm{(\na^{-1}N^1,N^2)(\tau)}_{L^r}+ \norm{\na^{\ell}(N^1,N^2)(\tau)}_{L^q}\right)\,d\tau
\end{align}
and
\begin{align}\label{u l}
\norm{\na^{\ell}u(t)}_{L^q}&\ls(1+t)^{-\frac{3}{2}\left(\frac{1}{p}-\frac{1}{q}\right)-\frac\ell2}\left(\norm{(\na^{-1}\varrho_0,u_0)}_{L^p}+ \norm{\na^{\ell}(\varrho_0,u_0)}_{L^q}\right) \nonumber\\
&\quad+\int_0^t(1+t-\tau)^{-\frac{3}{2}\left(\frac{1}{r}-\frac{1}{q}\right)-\frac\ell2}\left(\norm{(\na^{-1}N^1,N^2)(\tau)}_{L^r}+ \norm{\na^{\ell}(N^1,N^2)(\tau)}_{L^q}\right)\,d\tau.
\end{align}
\end{Proposition}

Finally, we record the following estimates of nonlinear terms appeared in \eqref{varrho l-1}--\eqref{u l}.
\begin{lemma}\label{mingti liu}
It holds that for $1< r\le 2$,
\begin{align}\label{N Lr}
\norm{(\na^{-1}N^1,N^2)}_{L^r}
&\ls\left(\de+\norm{ \rho_s-\bar\rho }_{W^{1,r}}\right)\left(\norm{u}_{L^\infty}+\norm{\na\varrho}_{H^2}+\norm{\na^2u}_{H^1}\right)\nonumber
\\&\quad+\norm{\varrho}_{L^2}\norm{\na^{3-\frac3r}u}_{L^2}+\norm{u}_{L^2}\norm{\na^{4-\frac3r}u}_{L^2};
\end{align}
for $0\le \ell\le 3/2$,
\begin{align}\label{N L2 l}
\norm{\na^{\ell}(N^1,N^2)}_{L^2}\ls\de\left(\norm{u}_{L^\infty}+\norm{\na\varrho}_{H^2}+\norm{\na^2u}_{H^2}\right);
\end{align}
and
\begin{align}\label{N Linfty}
\norm{(N^1,N^2)}_{L^\infty}\ls\de\left(\norm{u}_{L^\infty}+\norm{\na\varrho}_{H^2}+\norm{\na^2u}_{H^2}\right).
\end{align}
\end{lemma}
\begin{proof}
We will estimate the nonlinear terms term by term. First, for $1< r\le 2$, by the singular integral theory \cite{S}, the identity \eqref{r3} in the proof of Lemma \ref{A3} and H\"{o}lder's and Sobolev's inequalities, we obtain
\begin{align}
&\norm{\na^{-1}\diver((\varrho+\rho_s-\bar\rho)  u)}_{L^r}\ls\norm{ (\varrho+\rho_s-\bar\rho) u }_{L^r}\nonumber
\\&\quad\ls\norm{\varrho}_{L^2}\norm{u}_{L^{\frac{1}{1/r-1/2}}}+\norm{ \rho_s-\bar\rho}_{L^r}\norm{u}_{L^\infty} \ls\norm{\varrho}_{L^2}\norm{\na^{3-\frac3r}u}_{L^2}+\norm{ \rho_s-\bar\rho}_{L^r}\norm{u}_{L^\infty};\label{N yi}\\
&\norm{u\cdot\na u}_{L^r}\ls\norm{u}_{L^2}\norm{\na u}_{L^{\frac{1}{1/r-1/2}}}\ls\norm{u}_{L^2}\norm{\na^{4-\frac3r}u}_{L^2};\label{N er}\\
&\norm{\na\mathcal{R}}_{L^r}\ls\norm{h''(\rho_s)\varrho\na\varrho}_{L^r}+\norm{\mathcal{R}(h')\left(\na\varrho+\na\rho_s\right)}_{L^r}\nonumber\\
&\quad\ls\norm{ \varrho}_{L^{\frac{1}{1/r-1/2}}}\norm{\na \varrho}_{L^2}+\norm{\varrho}_{L^{\frac{1}{1/r-1/2}}}\norm{\varrho}_{L^\infty}\left(\norm{\na\varrho}_{L^2}+\norm{\na\rho_s}_{L^2}\right)\ls\de\norm{\na\varrho}_{H^2};\label{N san}\\
&\norm{\na((h'(\rho_s)-h'(\bar\rho))\varrho)}_{L^r}\ls\norm{h'(\rho_s)-h'(\bar\rho)}_{L^r}\norm{\na \varrho}_{L^\infty}+\norm{\na(h'(\rho_s)-h'(\bar\rho))}_{L^r}\norm{\varrho}_{L^\infty}\nonumber\\
&\quad\ls \norm{\rho_s-\bar\rho}_{W^{1,r}} \norm{\na\varrho}_{H^2};\label{N si}\\
&\norm{\left(\frac{1}{\varrho+\rho_s}-\frac{1}{\bar\rho}\right)\na^2u}_{L^r}\ls\norm{\frac{1}{\varrho+\rho_s}-\frac{1}{\bar\rho}}_{L^2}
\norm{\na^2u}_{L^{\frac{1}{1/r-1/2}}}\ls\de\norm{\na^2u}_{H^1}.\label{N wu}
\end{align}
These estimates \eqref{N yi}--\eqref{N wu} give \eqref{N Lr}.

Next, for $0\le \ell\le 3/2$, by the product estimates \eqref{product estimate}, the nonlinear estimates \eqref{A31} and Sobolev's inequality, we obtain
\begin{align}
&\norm{\na^{\ell}\diver((\varrho+\rho_s-\bar\rho)u)}_{L^2}\ls\norm{\varrho+\rho_s-\bar\rho}_{L^3}\norm{\na^{\ell+1}u}_{L^6}+\norm{\na^{\ell+1}(\varrho+\rho_s-\bar\rho)}_{L^2}\norm{u}_{L^\infty}\nonumber\\
&\quad\ls\de\left(\norm{u}_{L^\infty}+\norm{\na^{2}u}_{H^2}\right);\label{N shi}\\
&\norm{\na^{\ell}(u\cdot\na u)}_{L^2}\ls\norm{\na^{\ell}u}_{L^3}\norm{\na u}_{L^6}+\norm{u}_{L^3}\norm{\na^{\ell+1}u}_{L^6}\ls\de\norm{\na^2u}_{H^2};\\
&\norm{\na^{\ell}\na\mathcal{R}}_{L^2}\ls\de\left(\norm{\na\varrho}_{L^2}+\norm{\na^{\ell+1}\varrho}_{L^2}\right)\ls\de\norm{\na\varrho}_{H^2};\\
&\norm{\na^{\ell}\na((h'(\rho_s)-h'(\bar\rho))\varrho)}_{L^2}\nonumber\\
&\quad\ls\norm{\na^{\ell+1}(h'(\rho_s)-h'(\bar\rho))}_{L^3}\norm{\varrho}_{L^6}
+\norm{h'(\rho_s)-h'(\bar\rho)}_{L^\infty}\norm{\na^{\ell+1}\varrho}_{L^2}\ls\de\norm{\na\varrho}_{H^2};\\
&\norm{\na^{\ell}\left(\left(\frac{1}{\varrho+\rho_s}-\frac{1}{\bar\rho}\right)\na^2u\right)}_{L^2}\nonumber\\
&\quad\ls\norm{\na^{\ell}\left(\frac{1}{\varrho+\rho_s}-\frac{1}{\bar\rho}\right)}_{L^\infty}\norm{\na^2u}_{L^2}
+\norm{\frac{1}{\varrho+\rho_s}-\frac{1}{\bar\rho}}_{L^\infty}\norm{\na^{\ell+2}u}_{L^2}
\ls\de\norm{\na^2u}_{H^2}.\label{N shisi}
\end{align}
These estimates \eqref{N shi}--\eqref{N shisi} yield \eqref{N L2 l}.

Note that the estimate \eqref{N Linfty} can be obtained in a similar way.
\end{proof}

%%%%%%%%%%%%%%%%%%%%%%%%%%%%%%%%%%%%%%%%%%%%%%%%%%%%%%%%%%%%%%%%%%%%%%%%%%%%%
\subsection{Proof of Theorem \ref{decay}}
%%%%%%%%%%%%%%%%%%%%%%%%%%%%%%%%%%%%%%%%%%%%%%%%%%%%%%%%%%%%%%%%%%%%%%%%%%%%%

{{In this subsection, we will prove Theorem \ref{decay}. Let $k\ge3$ and $0\le\ell\le 3/2$.
Adding $\norm{\na^\ell(u,\na\Phi)}_{L^2}^2$ to both sides of the estimates \eqref{energy 5decay} of Proposition \ref{proposition decay}, we obtain
\begin{align}\label{232}
&\frac{d}{dt}\mathcal{E}_\ell^k+\norm{\na^{\ell}\varrho}_{H^{k-\ell}}^2+\norm{\na^{\ell}u}_{H^{k-\ell+1}}^2+\norm{ \na^{\ell}\na\Phi }_{L^2}^2\nonumber
\\&\quad\ls\norm{ \na^{\ell}\na\Phi }_{L^2}^2+\norm{ \na^{\ell}u }_{L^2}^2+\norm{\varrho}_{L^\infty}^2
+\norm{u}_{L^\infty}^2.
\end{align}
Note that $\norm{ \na^{\ell}\na\Phi }_{L^2}^2$ is equivalent to $\norm{ \na^{\ell-1}\varrho }_{L^2}^2$.
Hence, we deduce from \eqref{232} that
\begin{align}\label{200}
\frac{d}{dt}\mathcal{E}_\ell^k+\lambda\mathcal{E}_\ell^k \ls \norm{ \na^{\ell-1}\varrho }_{L^2}^2+\norm{ \na^{\ell}u }_{L^2}^2+\norm{\varrho}_{L^\infty}^2
+\norm{u}_{L^\infty}^2
\end{align}
 for some constant $\lambda>0$.
By the Gronwall inequality, we obtain
\begin{align}\label{2proof1}
&\norm{\na^{\ell-1} \varrho (t)}_{H^{k+1-\ell}}^2+\norm{\na^\ell u (t)}_{H^{k-\ell}}^2\ls e^{-\lambda t}\left(\norm{\na^\ell(\varrho_0,u_0) }_{H^{k-\ell}}^2+\norm{\na\Phi_0}_{L^2}^2\right)\nonumber\\
&\qquad\qquad+\int_0^te^{-\lambda(t-\tau)}\left( \norm{ \na^{\ell-1}\varrho(\tau) }_{L^2}^2+\norm{ \na^{\ell}u(\tau) }_{L^2}^2+\norm{\varrho(\tau)}_{L^\infty}^2
+\norm{u(\tau)}_{L^\infty}^2\right)\,d\tau.
\end{align}

We now prove \eqref{decay11}--\eqref{decay2}. So we let $k\ge 4$, $1< r<3/2$ and $1\le p<3/2$. For simplicity of notations, we denote
\begin{align}
K_0:=\norm{(\na^{-1}\varrho_0,u_0)}_{L^p}+\norm{(\varrho_0,u_0)}_{H^k}+\norm{\na\Phi_0}_{L^2},
\end{align}
and we define
\begin{align}
\zeta:=\frac{3}{2}\left(\frac{1}{\max\{p,r\}}-\frac{1}{2}\right)\text{ and }\tilde{\de}:=\de+\norm{\rho_s-\bar\rho }_{W^{1,r}}.
\end{align}
It turns out that we have to distinguish our arguments by the value of $r$.

{\it Case 1:  $6/5\le r<3/2$.} In this case, we define
\begin{align}\label{hua L}
\mathcal{L}(t):=\norm{\na^{1/2}\varrho(t) }_{L^2}+\norm{ \na^{3/2}u(t) }_{L^2}+\norm{\varrho(t)}_{L^\infty}
+\norm{u(t)}_{L^\infty},
\end{align}
\begin{align}\label{hua M}
\mathcal{M}(t):=\norm{\na^{1/2}\varrho(t) }_{H^{k-1/2}}+\norm{ \na^{3/2}u(t) }_{H^{k-3/2}},
\end{align}
and
\begin{align}\label{hua N}
\mathcal{N}(t):=\sup_{0\le\tau\le t}\left((1+\tau)^{\zeta+\frac{3}{4}}\left(\mathcal{L}(\tau)+\mathcal{M}(\tau) \right)
+(1+\tau)^{\zeta+\frac{1}{2}}\norm{\varrho(\tau)}_{L^2}+(1+\tau)^{\zeta}\norm{u(\tau)}_{L^2} \right).
\end{align}
We take $\ell=3/2$ in \eqref{2proof1} to have, in view of \eqref{hua L}--\eqref{hua M},
\begin{align}\label{hua M estimate}
\mathcal{M}^2(t) \ls e^{-\la t}K_0^2+\int_0^te^{-\la(t-\tau)}\mathcal{L}^2(\tau)\,d\tau.
\end{align}

We now estimate the time decay rates of $\mathcal{L}(t)$ by applying the linear decay estimates. By the estimates \eqref{varrho l-1} with $\ell=1/2$ and $q=2$ of Proposition \ref{mingti wu} and using the nonlinear estimates \eqref{N Lr}--\eqref{N L2 l}, in view of \eqref{hua L}--\eqref{hua M},  we obtain
\begin{align}\label{na 1/2 varrho nonlinear111}
&\norm{\na^{1/2}\varrho(t)}_{L^2}\ls(1+t)^{-\frac{3}{2}\left(\frac{1}{p}-\frac{1}{2}\right)-\frac{3}{4}}K_0+\int_0^t(1+t-\tau)^{-\frac{3}{2}\left(\frac{1}{r}-\frac{1}{2}\right)-\frac{3}{4}} \tilde\de (\mathcal{L}+\mathcal{M})(\tau)\,d\tau\nonumber\\
&   +\int_0^t(1+t-\tau)^{-\frac{3}{2}\left(\frac{1}{r}-\frac{1}{2}\right)-\frac{3}{4}}\left(\norm{\varrho(\tau)}_{L^2}\norm{\na^{3-\frac{3}{r}}u(\tau)}_{L^2}
+\norm{u(\tau)}_{L^2}\norm{\na^{4-\frac{3}{r}}u(\tau)}_{L^2}\right)\,d\tau .
\end{align}
By the interpolation, in view of \eqref{hua N}, we estimate
\begin{align}\label{104}
&\norm{\varrho(\tau)}_{L^2}\norm{\na^{3-\frac{3}{r}}u(\tau)}_{L^2}+\norm{u(\tau)}_{L^2}\norm{\na^{4-\frac{3}{r}}u(\tau)}_{L^2}\nonumber\\
&\quad\ls \norm{\na^{\vartheta}\varrho(\tau)}_{L^2}^{2-\frac2r}\norm{\na^{1/2}\varrho(\tau)}_{L^2}^{\frac2r-1}\norm{u(\tau)}_{L^2}^{\frac2r-1}
\norm{\na^{3/2} u(\tau)}_{L^2}^{2-\frac2r}+\norm{u(\tau)}_{L^2}\norm{\na^{4-\frac{3}{r}}u(\tau)}_{L^2} \nonumber\\
&\quad\ls \delta^{2-\frac2r} (1+\tau)^{\left(-\zeta-\frac{3}{4}\right)(\frac2r-1)} {\mathcal{N}(t)}^{\frac2r-1} \delta^{\frac2r-1} (1+\tau)^{\left(-\zeta-\frac{3}{4}\right)(2-\frac2r)} {\mathcal{N}(t)}^{2-\frac2r}+\delta(1+\tau)^{-\zeta-\frac{3}{4}} {\mathcal{N}(t)}\nonumber\\
&\quad\ls\de(1+\tau)^{-\zeta-\frac{3}{4}} \mathcal{N}(t).
\end{align}
Here we have used the facts that $\vartheta=(r-2)/(4r-4)\ge -1$ and $4-3/r\ge3/2$ since $r\ge6/5$.
Hence, plugging the estimates \eqref{104} into \eqref{na 1/2 varrho nonlinear111}, in view of \eqref{hua N}, we have
\begin{align}\label{na 1/2 varrho nonlinear}
\norm{\na^{1/2}\varrho(t)}_{L^2}
&\ls(1+t)^{-\frac{3}{2}\left(\frac{1}{p}-\frac{1}{2}\right)-\frac{3}{4} }K_0+\tilde{\de}\int_0^t(1+t-\tau)^{-\frac{3}{2}\left(\frac{1}{r}-\frac{1}{2}\right)-\frac{3}{4}}(1+\tau)^{-\zeta-\frac34}\mathcal{N}(t)\,d\tau\nonumber\\
&\ls (1+t)^{-\zeta-\frac{3}{4}}\left(K_0+\tilde{\de}\mathcal{N}(t)\right).
\end{align}
Here we have used the fact $\frac{3}{2}\left(\frac{1}{r}-\frac{1}{2}\right)+\frac{3}{4}>1$ since $r<\frac32$. Similarly, by the estimates \eqref{u l} with $\ell=3/2$ and $q=2$, and \eqref{varrho l-1}--\eqref{u l} with $\ell=0$ and $q=\infty$ of Proposition \ref{mingti wu} respectively, using the nonlinear estimates \eqref{N Lr}--\eqref{N Linfty}, we deduce
\begin{align}\label{u Linfty' nonlinear}
\norm{\na^{3/2}u(t)}_{L^2}+\norm{\varrho(t)}_{L^\infty}+\norm{u(t)}_{L^\infty}\ls (1+t)^{-\zeta-\frac{3}{4}}\left(K_0+\tilde{\de}\mathcal{N}(t)\right).
\end{align}
We thus deduce from \eqref{na 1/2 varrho nonlinear}--\eqref{u Linfty' nonlinear} that
\begin{align}\label{hua L estimate}
\mathcal{L}(t)\ls (1+t)^{-\zeta-\frac34}\left(K_0+\tilde{\de}\mathcal{N}(t)\right).
\end{align}
Now we substitute \eqref{hua L estimate} into \eqref{hua M estimate} to obtain
\begin{align}\label{hua M estimate11}
\mathcal{M}^2(t)& \ls e^{-\la t}K_0^2+\int_0^te^{-\lambda(t-\tau)} (1+\tau)^{-2\zeta-\frac32}(K_0^2+\tilde{\de}^2\mathcal{N}^2(t))\,d\tau.\nonumber\\
&\ls  (1+t)^{-2\zeta-\frac32}(K_0^2+\tilde{\de}^2\mathcal{N}^2(t)).
\end{align}

Finally, by the estimates \eqref{varrho l-1}--\eqref{u l} with $\ell=0$ and $q=2$ of Proposition \ref{mingti wu}, using the estimates \eqref{N Lr}--\eqref{N L2 l} with $\ell=0$ and \eqref{104}, in view of \eqref{hua N}, we obtain
\begin{align}\label{varrho L2 estimate}
\norm{\varrho(t)}_{L^2}
&\ls(1+t)^{-\frac{3}{2}\left(\frac{1}{p}-\frac{1}{2}\right)-\frac{1}{2} }K_0+\tilde{\de}\int_0^t(1+t-\tau)^{-\frac{3}{2}\left(\frac{1}{r}-\frac{1}{2}\right)-\frac{1}{2}}  (1+\tau)^{-\zeta-\frac34}\mathcal{N}(t)\,d\tau\nonumber\\
&\ls (1+t)^{-\zeta-\frac{1}{2}}\left(K_0+\tilde{\de}\mathcal{N}(t)\right)
\end{align}and
\begin{align}\label{u L2 estimate}
\norm{u(t)}_{L^2}
&\ls(1+t)^{-\frac{3}{2}\left(\frac{1}{p}-\frac{1}{2}\right) }K_0+\tilde{\de}\int_0^t(1+t-\tau)^{-\frac{3}{2}\left(\frac{1}{r}-\frac{1}{2}\right)}  (1+\tau)^{-\zeta-\frac34}\mathcal{N}(t)\,d\tau\nonumber\\
& \ls  (1+t)^{-\zeta}\left(K_0+\tilde{\de}\mathcal{N}(t)\right).
\end{align}
Note that we have used the fact $\zeta+\frac34>1$ since $p,r<3/2$ so that
\begin{align}
 \int_0^t(1+t-\tau)^{-\frac{3}{2}\left(\frac{1}{r}-\frac{1}{2}\right)-\frac{1}{2}}  (1+\tau)^{-\zeta-\frac34}\,d\tau\ls(1+t)^{-\frac{3}{2}\left(\frac{1}{r}-\frac{1}{2}\right)-\frac{1}{2}}
\end{align}
and
\begin{align}
\int_0^t(1+t-\tau)^{-\frac{3}{2}\left(\frac{1}{r}-\frac{1}{2}\right)}  (1+\tau)^{-\zeta-\frac34}\,d\tau\ls(1+t)^{-\frac{3}{2}\left(\frac{1}{r}-\frac{1}{2}\right)}.
\end{align}
By the definition \eqref{hua N} of $\mathcal{N}(t)$, we deduce from \eqref{hua L estimate}--\eqref{u L2 estimate} that
\begin{align}
\mathcal{N} (t)\ls K_0+\tilde{\de}\mathcal{N}(t).
\end{align}
This implies
\begin{align}
\mathcal{N}(t)\ls K_0
\end{align}
since $\tilde{\de}$ is small by Proposition \ref{prop}. This in turn together with the interpolation gives \eqref{decay11}--\eqref{decay2} for $6/5\le r<3/2$ by taking $C_0=K_0$.

{\it Case 2:  $1< r<6/5$.} In this case, we define
\begin{align}\label{hua H}
\mathcal{H}(t):=\norm{\varrho(t) }_{L^2}+\norm{ \na u(t) }_{L^2}+\norm{\varrho(t)}_{L^\infty}
+\norm{u(t)}_{L^\infty},
\end{align}
\begin{align}\label{hua J}
\mathcal{J}(t):=\norm{\varrho(t) }_{H^{k}}+\norm{ \na u(t) }_{H^{k-1}},
\end{align}
and
\begin{align}\label{hua K}
\mathcal{K}(t):=\sup_{0\le\tau\le t}\left((1+\tau)^{\zeta+\frac{1}{2}}\left(\mathcal{H}(\tau)+\mathcal{J}(\tau) \right)
\right).
\end{align}
We take $\ell=1$ in \eqref{2proof1} to have, in view of \eqref{hua H}--\eqref{hua J},
\begin{align}\label{301}
\mathcal{J}^2(t) \ls e^{-\la t}K_0^2+\int_0^te^{-\la(t-\tau)}\mathcal{H}^2(\tau)\,d\tau.
\end{align}

We now estimate the time decay rates of $\mathcal{H}(t)$ by applying the linear decay estimates. By the estimates \eqref{varrho l-1} with $\ell=0$ and $q=2$ of Proposition \ref{mingti wu} and using the nonlinear estimates \eqref{N Lr}--\eqref{N L2 l}, in view of \eqref{hua H}--\eqref{hua J},  we obtain
\begin{align}\label{varrho L2 nonlinear}
&\norm{\varrho(t)}_{L^2}\ls(1+t)^{-\frac{3}{2}\left(\frac{1}{p}-\frac{1}{2}\right)-\frac{1}{2}}K_0+\int_0^t(1+t-\tau)^{-\frac{3}{2}\left(\frac{1}{r}-\frac{1}{2}\right)-\frac{1}{2}} \tilde\de (\mathcal{H}+\mathcal{J})(\tau)\,d\tau\nonumber\\
&   +\int_0^t(1+t-\tau)^{-\frac{3}{2}\left(\frac{1}{r}-\frac{1}{2}\right)-\frac{1}{2}}\left(\norm{\varrho(\tau)}_{L^2}\norm{\na^{3-\frac{3}{r}}u(\tau)}_{L^2}
+\norm{u(\tau)}_{L^2}\norm{\na^{4-\frac{3}{r}}u(\tau)}_{L^2}\right)\,d\tau .
\end{align}
Note that
\begin{align}\label{2323}
\norm{\varrho}_{L^2}\norm{\na^{3-\frac3r}u}_{L^2}+\norm{u}_{L^2}\norm{\na^{4-\frac3r}u}_{L^2}\ls\de\left(\norm{\varrho }_{L^2}+\norm{ \na u}_{H^1}\right)\ls\de \mathcal{J}.
\end{align}
Hence, we have
\begin{align}\label{na 0 varrho nonlinear}
\norm{\varrho(t)}_{L^2}&\ls(1+t)^{-\frac{3}{2}\left(\frac{1}{p}-\frac{1}{2}\right)-\frac{1}{2}}K_0+\int_0^t(1+t-\tau)^{-\frac{3}{2}\left(\frac{1}{r}-\frac{1}{2}\right)-\frac{1}{2}} \tilde\de (\mathcal{H}+\mathcal{J})(\tau)\,d\tau\nonumber\\
&\ls(1+t)^{-\frac{3}{2}\left(\frac{1}{p}-\frac{1}{2}\right)-\frac{1}{2} }K_0+\tilde{\de}\int_0^t(1+t-\tau)^{-\frac{3}{2}\left(\frac{1}{r}-\frac{1}{2}\right)-\frac{1}{2}}  (1+\tau)^{-\zeta-\frac12}\mathcal{K}(t)\,d\tau\nonumber\\
&\ls (1+t)^{-\zeta-\frac{1}{2}}\left(K_0+\tilde{\de}\mathcal{K}(t)\right).
\end{align}
Here we have used the fact $\frac{3}{2}\left(\frac{1}{r}-\frac{1}{2}\right)+\frac{1}{2}>1$ since $r<6/5$. Similarly, by the estimates \eqref{u l} with $\ell=1$ and $q=2$, and \eqref{varrho l-1}--\eqref{u l} with $\ell=0$ and $q=\infty$ of Proposition \ref{mingti wu} respectively, using the nonlinear estimates \eqref{N Lr}--\eqref{N Linfty}, we deduce
\begin{align}\label{u Linfty nonlinear}
\norm{\na u(t)}_{L^2}+\norm{\varrho(t)}_{L^\infty}+\norm{u(t)}_{L^\infty}\ls (1+t)^{-\zeta-\frac{1}{2}}\left(K_0+\tilde{\de}\mathcal{K}(t)\right).
\end{align}
We thus deduce from \eqref{na 0 varrho nonlinear}--\eqref{u Linfty nonlinear} that
\begin{align}\label{hua H estimate}
\mathcal{H}(t)\ls  (1+t)^{-\zeta-\frac12}\left(K_0+\tilde{\de}\mathcal{K}(t)\right).
\end{align}
Now we substitute \eqref{hua H estimate} into \eqref{301} to obtain
\begin{align}\label{hua J estimate}
\mathcal{J}^2(t)&\ls e^{-\la t}K_0^2+\int_0^te^{-\lambda(t-\tau)} (1+\tau)^{-2\zeta-1}(K_0^2+\tilde{\de}^2\mathcal{K}^2(t))\,d\tau.\nonumber\\
&\ls  (1+t)^{-2\zeta-1}(K_0^2+\tilde{\de}^2\mathcal{K}^2(t)).
\end{align}
By the definition \eqref{hua K} of $\mathcal{K}(t)$, we deduce from \eqref{hua H estimate}--\eqref{hua J estimate} that
\begin{align}
\mathcal{K} (t)\ls K_0+\tilde{\de}\mathcal{K}(t).
\end{align}
This implies, since $\tilde{\de}$ is small,
\begin{align}\label{444}
\mathcal{K}(t)\ls K_0.
\end{align}

Finally, by the estimates \eqref{u l} with $\ell=0$ and $q=2$ of Proposition \ref{mingti wu}, using the estimates \eqref{N Lr}--\eqref{N L2 l} with $\ell=0$ and \eqref{2323}, in view of \eqref{hua K}, by \eqref{444}, we obtain
\begin{align}\label{u estimate}
\norm{u(t)}_{L^2}
&\ls(1+t)^{-\frac{3}{2}\left(\frac{1}{p}-\frac{1}{2}\right) }K_0+\tilde{\de}\int_0^t(1+t-\tau)^{-\frac{3}{2}\left(\frac{1}{r}-\frac{1}{2}\right)}  (1+\tau)^{-\zeta-\frac34}\mathcal{K}(t)\,d\tau\nonumber\\
& \ls   K_0 (1+t)^{-\zeta}.
\end{align}
Note that \eqref{444} implies
\begin{align}\label{101}
\norm{\varrho(t)}_{H^k}+\norm{\na u(t)}_{H^{k-1}}\ls K_0(1+t)^{-\zeta-\frac{1}{2}}.
\end{align}
That is, we have proved \eqref{decay11} for $\ell=0$ and \eqref{decay1} for $0\le \ell\le 1$ by the interpolation. To prove the remaining decay estimates in \eqref{decay11}--\eqref{decay2}, we may now employ the arguments used in {\it Case 1}. Indeed, since $0<3-3/r<1$ and $4-3/r>1$, by the interpolation, we deduce from \eqref{u estimate}--\eqref{101} that
\begin{align}\label{103}
&\norm{\varrho(\tau)}_{L^2}\norm{\na^{3-\frac{3}{r}}u(\tau)}_{L^2}+\norm{u(\tau)}_{L^2}\norm{\na^{4-\frac{3}{r}}u(\tau)}_{L^2}\nonumber\\
&\quad\ls K_0^2(1+\tau)^{-\zeta-\frac{1}{2}} (1+\tau)^{-\zeta-\frac{3-3/r}{2}}+K_0^2(1+\tau)^{-\zeta} (1+\tau)^{-\zeta-\frac12} \nonumber\\
&\quad\ls K_0^2(1+\tau)^{-\zeta-\frac{3}{4}}.
\end{align}
Here we have used the fact $2\zeta+1/2\ge \zeta+3/4$ since $p,r<3/2$. So by replacing the estimates \eqref{104} by the estimates \eqref{103} and then reproducing the arguments of {\it Case 1}, we may derive
\begin{align}
\mathcal{N}(t)\ls K_0+K_0^2,
\end{align}
where $\mathcal{N}(t)$ is defined by \eqref{hua N}. This in turn together with the interpolation gives \eqref{decay11}--\eqref{decay2} for $1< r<6/5$ by taking $C_0=K_0+K_0^2$.

Now in view of these two cases, the proof of Theorem \ref{decay} is completed.}}\hfill$\Box$

\appendix

%%%%%%%%%%%%%%%%%%%%%%%%%%%%%%%%%%%%%%%%%%%%%%%
\section{Analytic tools}\label{1section_appendix}
%%%%%%%%%%%%%%%%%%%%%%%%%%%%%%%%%%%%%%%%%%%%%%%
We recall the Sobolev interpolation of the Gagliardo-Nirenberg inequality.
\begin{lemma}\label{A1}
Let $2\le p\le \infty$ and $\alpha,\beta,\gamma\in \mathbb{R}$. Then we have
\begin{align}\label{A.1}
\norm{\na^\alpha f}_{L^p}\lesssim \norm{ \na^\beta f}_{L^2}^{1-\theta}
\norm{ \na^\gamma f}_{L^2}^{\theta}.
\end{align}
Here $0\le \theta\le 1$ (if $p=\infty$, then we require that $0<\theta<1$) and $\alpha$ satisfy
\begin{align*}
\alpha+3\left(\frac12-\frac{1}{p}\right)=\beta(1-\theta)+\gamma\theta.
\end{align*}
\end{lemma}
\begin{proof}
For the case $2\le p<\infty$, we refer to Lemma 2.4 in \cite{GW}; for the case $p=\infty$, we refer to Exercise 6.1.2 in \cite{Gla}.
\end{proof}

We then recall the following commutator and product estimates:
\begin{lemma}\label{A2}
Let $l\ge 0$ and define the commutator
\begin{align}\label{commutator}
\left[\na^l,g\right]h=\na^l(gh)-g\na^lh.
\end{align}
Then we have
\begin{align}\label{commutator estimate}
\norm{\left[\na^l,g\right]h}_{L^{p_0}} \ls\norm{\na g}_{L^{p_1}}
\norm{\na^{l-1}h}_{L^{p_2}} +\norm{\na^l g}_{L^{p_3}}\norm{h}_{L^{p_4}}.
\end{align}
In addition, we have that for $l\ge0$,
\begin{align}\label{product estimate}
\norm{\na^l(gh)}_{L^{p_0}} \ls\norm{g}_{L^{p_1}}
\norm{\na^{l}h}_{L^{p_2}} +\norm{\na^l g}_{L^{p_3}} \norm{ h}_{L^{p_4}}.
\end{align}
Here $p_0,p_2,p_3\in(1,\infty)$ and
\begin{align*}
\frac1{p_0}=\frac1{p_1}+\frac1{p_2}=\frac1{p_3}+\frac1{p_4}.
\end{align*}
\end{lemma}
\begin{proof}
We refer to Lemma 3.1 in \cite{J}.
\end{proof}

Lastly, we record the estimates of the remainder $\mathcal{R}$ defined by \eqref{hhhggg}.
\begin{lemma}\label{A3}
Let $\mathcal{R}$ be defined by \eqref{hhhggg}. Then we have that for $l\ge1$,
\begin{align}\label{A31}
\norm{\na^l\mathcal{R}}_{L^2}\ls\de\left(\norm{\na\varrho}_{L^2}+\norm{\na^l\varrho}_{L^2}\right)
\end{align}
and
\begin{align}\label{A32}
\norm{\na^l\mathcal{R}}_{L^2}\ls\de\left(\norm{\varrho}_{L^\infty}+\norm{\na^l\varrho}_{L^2}\right).
\end{align}
Here $``\ls"$ stands for $``\le C"$ with the constant $C$ depending on the function $h$, the upper and lower bounds of $\rho_s$ and $\norm{\nabla \rho_s}_{H^l}$.
\end{lemma}
\begin{proof}
We only prove \eqref{A31}, while \eqref{A32} can be proved similarly with minor modifications.
We may view $\mathcal{R}$ as an operator over $h$, $i.e.$, we define the operator $\mathcal{R}(f)$ of the smooth function $f$:
\begin{align}\label{r1}
\mathcal{R}(f)&:=\int_{\rho_s}^{\varrho+\rho_s}f''(s)(\varrho+\rho_s-s)\,ds\nonumber\\
&\equiv\int_{0}^{\varrho}f''(\varrho+\rho_s-\tau)\tau\,d\tau.
\end{align}
Then $\mathcal{R}=\mathcal{R}(h)$. It is clear from the definition \eqref{r1} that
\begin{align}\label{r2}
\mathcal{R}(f)=O(\varrho^2).
\end{align}
Moreover, taking the spatial derivative of \eqref{r1} yields
\begin{align}\label{r3}
\na\mathcal{R}(f)&=f''(\rho_s)\varrho\na\varrho+\int_{0}^{\varrho}f'''(\varrho+\rho_s-\tau)\tau\,d\tau\left(\na\varrho+\na\rho_s\right)\nonumber\\
&\equiv f''(\rho_s)\varrho\na\varrho+\mathcal{R}(f')\left(\na\varrho+\na\rho_s\right).
\end{align}
Hence, by H\"older's and Sobolev's inequalities, we have
\begin{align}\label{r4}
\norm{\na\mathcal{R}(f)}_{L^2}&\ls\norm{f''(\rho_s)\varrho\na\varrho}_{L^2}+\norm{\mathcal{R}(f')\left(\na\varrho+\na\rho_s\right)}_{L^2}\nonumber\\
&\ls\norm{f''(\rho_s)}_{L^\infty}\norm{\varrho}_{L^\infty}\norm{\na\varrho}_{L^2}+\norm{\varrho}_{L^\infty}\norm{\varrho}_{L^6}\left(\norm{\na\varrho}_{L^3}
+\norm{\na\rho_s}_{L^3}\right)\nonumber\\
&\ls\norm{\varrho}_{L^\infty}\norm{\na\varrho}_{L^2}\ls \delta\norm{\na\varrho}_{L^2}.
\end{align}
Since $\mathcal{R}=\mathcal{R}(h)$, we deduce \eqref{A31} for $l=1$.

Now for $l\ge2$, by the identity \eqref{r3} and the product estimates \eqref{product estimate} of Lemma \ref{A2}, we obtain
\begin{align}\label{r5}
&\norm{\na^{l}\mathcal{R}(f)}_{L^2}=\norm{\na^{l-1}\left(\na\mathcal{R}(f)\right)}_{L^2}\nonumber\\
&\quad\ls\norm{\na^{l-1}\left(f''(\rho_s)\varrho\na\varrho\right)}_{L^2}+\norm{\na^{l-1}\left(\mathcal{R}(f')\left(\na\varrho+\na\rho_s\right)\right)}_{L^2}\nonumber\\
&\quad\ls\norm{\na^{l-1}(f''(\rho_s))}_{L^3}\norm{\varrho\na\varrho}_{L^6}+\norm{f''(\rho_s)}_{L^\infty}\norm{\na^{l-1}\left(\varrho\na\varrho\right)}_{L^2}
\nonumber\\
&\quad\quad+\norm{\mathcal{R}(f')}_{L^\infty}\norm{\na^{l}\varrho}_{L^2}+\norm{\mathcal{R}(f')}_{L^6}\norm{\na^{l}\rho_s}_{L^3}\nonumber\\
&\quad\quad+\left(\norm{\na\varrho}_{L^\infty}+\norm{\na\rho_s}_{L^\infty}\right)\norm{\na^{l-1}\mathcal{R}(f')}_{L^2}\nonumber\\
&\quad\ls\norm{\varrho}_{L^6}\norm{\na\varrho}_{L^\infty}+\norm{\varrho}_{L^\infty}\norm{\na^{l}\varrho}_{L^2}+\norm{\na^{l-1}\varrho}_{L^6}\norm{\na\varrho}_{L^3}\nonumber\\
&\quad\quad+\norm{\varrho}_{L^\infty}^2\norm{\na^{l}\varrho}_{L^2}+\norm{\varrho}_{L^\infty}\norm{\varrho}_{L^6}+\norm{\na^{l-1}\mathcal{R}(f')}_{L^2}
\nonumber\\
&\quad\ls\de\left(\norm{\na\varrho}_{L^2}+\norm{\na^{l}\varrho}_{L^2}\right)+\norm{\na^{l-1}\mathcal{R}(f')}_{L^2}.
\end{align}
By this recursive inequality \eqref{r5}, we obtain that for $l\ge2$,
\begin{align}\label{r6}
\norm{\na^l\mathcal{R}}_{L^2}\equiv\norm{\na^l\mathcal{R}(h)}_{L^2}&\ls\de\sum^{l}_{\ell=1}\norm{\na^\ell\varrho}_{L^2}+\norm{\nabla \mathcal{R}(h^{(l-1)})}_{L^2}
\nonumber\\
&\ls\de\left(\norm{\na\varrho}_{L^2}+\norm{\na^l\varrho}_{L^2}\right).
\end{align}
Here in the last inequality we have used the inequality \eqref{r4} for $f=h^{(l-1)}$. This proves \eqref{A31} for $l\ge 2$, and the proof of the lemma is completed.
\end{proof}

%%%%%%%%%%%%%%%%%%%%%%%%%%%%%%%%%%%%%%%%%%%%%%%
\section*{Acknowledgements}
%%%%%%%%%%%%%%%%%%%%%%%%%%%%%%%%%%%%%%%%%%%%%%%
The authors are deeply grateful to the referees for the valuable comments and suggestions.

\end{document}